\documentclass[hidelinks,onefignum,onetabnum]{siamart251216}
\usepackage{amssymb,amsmath,graphicx,bm,color,booktabs,comment,stmaryrd}
\usepackage{caption,subcaption,float,mathtools,epstopdf}
\numberwithin{figure}{section}
\numberwithin{table}{section}
\newsiamthm{thm}{Theorem}
\newsiamthm{axiom}{Axiom}
\newsiamthm{cor}{Corollary}
\newsiamthm{lem}{Lemma}
\newsiamthm{prop}{Proposition}
\newsiamthm{defn}{Definition}
\newsiamremark{rem}{Remark}

\def\theequation{\thesection.\arabic{equation}}

\newcommand{\ie}{{\it{i.e.}}}

\newcommand{\eg}{{\it{e.g.}}}

\newcommand{\diag}{\mbox{diag}}

\headers{LBM for Hyperbolic Systems}{Y. Zhang, P. Lin, J. Zhao, and W. Zhao}
\title{A Lattice Boltzmann Method with Adaptive Relaxation Parameter and Lax--Friedrichs-Type Equilibrium for Hyperbolic Systems}
\author{Yifan Zhang\thanks{Department of Applied Mathematics, University of Science and Technology Beijing, Beijing 100083, China (\email{yifanSXY69@163.com}).}
\and Ping Lin\thanks{Division of Mathematics, University of Dundee, Dundee DD1 4HN, United Kingdom (\email{p.lin@dundee.ac.uk}).}
\and Jin Zhao\thanks{Academy for Multidisciplinary Studies, Capital Normal University, Beijing 100048, China (\email{zjin@cnu.edu.cn}). Corresponding author.}
\and Weifeng Zhao\thanks{Department of Applied Mathematics, University of Science and Technology Beijing, Beijing 100083, China (\email{wfzhao@ustb.edu.cn}). Corresponding author.}}

\begin{document}
\maketitle

\begin{abstract}
In this paper, a lattice Boltzmann method with an adaptive relaxation parameter and a Lax--Friedrichs-type equilibrium is proposed for hyperbolic systems with source terms. The equilibrium distribution recovers the conservative variables and physical fluxes while incorporating dissipation determined by characteristic-speed bounds. To balance accuracy and robustness, the relaxation parameter is selected from a local smoothness indicator based on characteristic projections. In smooth regions, the parameter approaches the low-dissipation limit, retaining second-order accuracy; near discontinuities, it is automatically reduced to introduce localized dissipation and suppress nonphysical oscillations. The stabilization acts directly through the local collision step and preserves the standard collide-and-stream structure, without a posteriori recomputation or interface-based limiting. Maxwell iteration establishes second-order consistency in smooth regions, and a weighted $L^2$-stability estimate is proved for linear hyperbolic systems with periodic boundary conditions under a standard CFL condition. Numerical experiments for scalar advection, Euler, shallow-water, and reactive Euler equations demonstrate the expected accuracy for smooth solutions and robust resolution of challenging one- and two-dimensional discontinuous problems, including wet--dry fronts and reactive discontinuities. A large-scale simulation of a 2D cellular detonation further demonstrates the capability of the method to resolve long-time multidimensional shock--reaction interactions.

\end{abstract}



\section{Introduction}

Hyperbolic systems of conservation laws arise in a wide range of scientific and engineering applications, including gas dynamics, shallow-water flows, magnetohydrodynamics, traffic flow, and chemically reactive flows. A major difficulty in the numerical approximation of such systems is the coexistence of smooth wave propagation and nonsmooth structures, such as shocks, contact discontinuities, rarefaction waves, wet--dry fronts, and material interfaces. In smooth regions, a desirable numerical method should retain high accuracy and low numerical dissipation, whereas near discontinuities sufficient numerical dissipation is required to suppress spurious oscillations and maintain robustness. Designing schemes that can balance these competing requirements remains a central issue in the computation of hyperbolic systems \cite{LeVeque1992,Toro2009}.

Over the past decades, many high-resolution methods have been developed for hyperbolic conservation laws. Total variation diminishing schemes and flux-limiter methods provide robust non-oscillatory approximations near discontinuities \cite{Harten1983}. Essentially non-oscillatory and weighted essentially non-oscillatory schemes improve shock resolution through nonlinear stencil selection and nonlinear reconstruction procedures \cite{HartenEngquistOsherChakravarthy1987,JiangShu1996,Shu1998}. Discontinuous Galerkin methods combine high-order polynomial approximations with elementwise conservation and have been widely used for convection-dominated problems \cite{CockburnShu2001}. These methods are highly effective, but for nonlinear systems and multidimensional balance laws their implementation often involves characteristic decomposition, nonlinear reconstruction, limiters, approximate Riemann solvers, or carefully designed numerical fluxes. Such ingredients may increase the algorithmic complexity, especially when source terms, multidimensional wave interactions, or reactive effects are present.

The lattice Boltzmann method (LBM) provides an alternative kinetic framework for approximating macroscopic equations through the evolution of distribution functions on discrete lattices. Its local collision--streaming structure, explicit time evolution, and natural parallelism have made it attractive for incompressible flows, complex fluids, porous media, heat transfer, and other transport phenomena \cite{ChenDoolen1998,Succi2001,Mohamad2011,GuoShu2013}. The macroscopic behavior of LBM has been studied through Chapman--Enskog expansion, asymptotic analysis, Maxwell iteration, and equivalent-equation approaches \cite{LallemandLuo2000,JunkKlarLuo2005,Dubois2008,YongZhaoLuo2016,ZhaoYong2017}. These analyses show that the accuracy, stability, and dissipation properties of lattice Boltzmann schemes depend strongly on the equilibrium distribution and relaxation parameters.

The use of LBM for hyperbolic conservation laws has developed along closely related kinetic and relaxation routes. Relaxation systems and kinetic BGK models provide the analytical basis for replacing a nonlinear conservation law by discrete-velocity transport coupled through relaxation \cite{JinXin1995,Bouchut1999,Bouchut2003,AregbaDriolletNatalini2000}, while shock-oriented and compressible LBMs demonstrate the potential of kinetic discretizations for strongly compressible flows \cite{WatariTsutahara2003,YanChikatamarlaKarlin2009,XuZhangGanChenYu2012}. Graille constructed a vectorial LBM for general one-dimensional hyperbolic systems and related it to an explicit discretization of the Jin--Xin relaxation system; the equivalent equations and scalar stability conditions were also analyzed \cite{Graille2014}. The macroscopic finite-difference representation of general LBMs was subsequently clarified in \cite{BellottiGrailleMassot2022}. For hyperbolic equations with source terms, Anandan and Raghurama Rao studied explicit and semi-implicit vector-kinetic LBMs, deriving their macroscopic finite-difference forms and consistency and examining H-inequality, total-variation, positivity, source balancing, and multidirectional upwinding \cite{AnandanRao2024}.

Rigorous stability and convergence results for LBMs applied to hyperbolic equations remain strongly dependent on the model and the available kinetic structure. Guillon, H\'elie, and Helluy established entropy stability of an over-relaxed vectorial LBM under the existence of suitable convex kinetic entropies and also introduced an augmented equivalent-system analysis of stability \cite{GuillonHelieHelluy2024}. Building on the kinetic-entropy framework, Bellotti, Helluy, and Navoret constructed fourth-order entropy-stable LB schemes for multidimensional nonlinear systems of conservation laws by composing time-symmetric second-order steps and adapting the relaxation parameter locally at each stage \cite{BellottiHelluyNavoret2025}. From a different perspective, Aregba-Driollet and Bellotti derived monotonicity and \(L^\infty\)-stability conditions for two-relaxation-times LB schemes and proved convergence to the weak entropy solution of a scalar nonlinear conservation law \cite{AregbaDriolletBellotti2026}. Their analysis also shows that the additional relaxation freedom can reduce numerical diffusion relative to the corresponding BGK scheme. This convergence result is rigorous but is restricted to a scalar equation; extending comparable estimates to general nonlinear hyperbolic systems remains difficult.

Recent work has also concentrated on robust stabilization near discontinuities. Kozhanova et al. combined a robust first-order LB scheme and a second-order LB scheme through an a posteriori MOOD strategy with smoothness and positivity detectors \cite{KozhanovaZhaoLoubereBoivin2025}. Wissocq, Liu, and Abgrall developed a two-dimensional vectorial LBM for the Euler equations \cite{WissocqLiuAbgrall2025}. They identified a CFL condition under which the first-order scheme is convex preserving and blended first- and second-order updates with convex limiters designed to preserve either positivity of density and internal energy or local maximum principles. Their computations demonstrated robustness for strong discontinuities and near-vacuum states. These developments demonstrate the potential of vectorial LBMs for hyperbolic systems. Nevertheless, robust computations near strong discontinuities generally require additional entropy control, switching, a posteriori detection, or convex/interface limiting. This motivates a complementary mechanism that controls dissipation directly through the local collision relaxation while retaining the standard collide-and-stream update.

In this paper, we propose an efficient LBM with adaptive relaxation parameter for hyperbolic systems with source terms. A Lax--Friedrichs-type equilibrium is constructed so that its moments recover the conservative variables and physical fluxes while characteristic-speed bounds provide a transparent baseline dissipation. The relaxation parameter is then selected locally from smoothness indicators formed by characteristic projections. This construction directly addresses the compromise inherent in a fixed relaxation parameter: in smooth regions the parameter approaches $2$, recovering the low-dissipation second-order regime, whereas near discontinuities it decreases automatically and supplies the additional dissipation needed to control nonphysical oscillations. The proposed stabilization does not require the construction of convex kinetic entropies. Moreover, characteristic projections enter only the smoothness detector; the collide-and-stream update itself requires neither nonlinear reconstruction nor a Riemann-solver-based numerical flux, and no a posteriori recomputation or interface-based convex limiter is introduced.

The analysis and experiments support this construction in several complementary ways. Maxwell iteration is used to show its second-order accuracy in both space and time for smooth solutions. A matrix-weighted $L^2$-stability estimate is established for linear hyperbolic systems with spatially and temporally varying relaxation parameters. The numerical study verifies second-order accuracy for smooth scalar and Euler problems. It also demonstrates localized stabilization for discontinuous advection, shock tubes, shock--entropy-wave interaction, wet--dry fronts, the two-dimensional Euler Riemann problem, and reactive flows. In particular, a large-scale 2D cellular-detonation simulation demonstrates the ability of the method to capture a corrugated leading front, transverse-wave interactions, and a persistent quasi-periodic cellular structure over long times. Comparisons with fixed relaxation, a limiter-based vectorial LBM, and WENO reference solutions illustrate the balance achieved between accuracy in smooth regions and robustness near nonsmooth structures. 

We emphasize that the proposed strategy is modular: once an admissible equilibrium and directional characteristic information are available, the same adaptive-relaxation principle can be incorporated into other vectorial or multi-relaxation lattice Boltzmann formulations. It is therefore potentially applicable to broader hyperbolic and balance-law systems, including magnetohydrodynamics, elastodynamics, multispecies transport, and more general reactive-flow models. The local collide-and-stream structure is also compatible with parallel implementations and local mesh adaptation. For problems involving extreme states, combining the present smoothness-based dissipation control with invariant-domain or positivity-preserving mechanisms is a natural extension.

The remainder of this paper is organized as follows. Section 2 presents the lattice Boltzmann formulation, the Lax--Friedrichs-type equilibrium, and the adaptive relaxation parameter. Section 3 gives the consistency analysis based on Maxwell iteration. Section 4 establishes the weighted $L^2$-stability result for linear hyperbolic systems. Section 5 reports the numerical experiments. Section 6 concludes the paper.

\section{LBM for hyperbolic systems}
\setcounter{equation}{0}

\subsection{LBM}
Consider a general hyperbolic system with a source term:
\begin{equation}\label{21}
\partial_t
U
+ \sum_{j=1}^d \partial_{x_j} F_j(U)
= S(U),
\end{equation}
where $U \in \mathbb R^m$ is the unknown $m$-vector of $(\bm x, t)\equiv(x_1,x_2,...,x_d,t) \in \mathbb R^d \times [0, +\infty)$, $F_j:=F_j(U) \in \mathbb R^{m}$ are the fluxes, and $S(U) \in \mathbb R^{m}$ is the source term.
The LBM for the hyperbolic system \eqref{21} can be written as \cite{Graille2014,GuoZhengShi2002}
\begin{equation}\label{22}
\begin{split}
\bm{f}_k\bigl(\bm{x}_{i} + \bm{e}_k \delta_x,t_{n+1}\bigr)
=
& \bm{f}_k\bigl(\bm{x}_{i}, t_n\bigr)
-\omega
\left[
\bm{f}_k\bigl(\bm{x}_{i},t_n\bigr)-\bm{f}_k^{(eq)}\bigl(\bm{x}_{i},t_n\bigr)
\right] \\
& +  \frac{1}{2} \delta_t \bm{s}_k\bigl(\bm{x}_{i},t_n\bigr)
  + \frac{1}{2} \delta_t \bm{s}_k\bigl(\bm{x}_{i} + \bm{e}_k \delta_x,t_{n+1}\bigr),
\end{split}
\end{equation}
where $\bm f_k(\bm{x}_{i},t_n) \in \mathbb R^{m}$ is the vector-valued distribution function at lattice node $\bm{x}_{i}$ and time $t_n$, $\bm{e}_k$ is the discrete velocity, $\delta_x$ is the lattice size, $t_{n+1} = t_n + \delta_t$, $\delta_t$ is the time step, $\omega $ is the relaxation parameter, $\bm{f}_k^{(eq)}=\bm{f}_k^{(eq)}(U)$ is the equilibrium, and $\bm{s}_k$ is related to the source term. The macroscopic quantity $U$ is computed via the distribution functions as
\begin{equation}\label{23}
U
=
\sum_{k=1}^N \bm{f}_k.
\end{equation}
To recover the macroscopic equation \eqref{21}, it is required that
\begin{equation}\label{24}
\sum_{k=1}^N \bm f_k^{(eq)} =U, \quad \sum_{k=1}^N v_{kj} \bm f_k^{(eq)} = F_j(U), \quad  \sum_{k=1}^N \bm s_k = S(U),
\end{equation}
where $v_{kj}$ is the $j$-th component of $\bm v_k = \lambda \bm e_k$ with $\lambda = \delta_x / \delta_t$.
This can be achieved by choosing $\bm f_k^{(eq)}$ and $\bm s_k$ as
\begin{equation}\label{25}
\bm f_k^{(eq)} = a_k U + \sum_{j=1}^d b_{kj} F_j(U),
\end{equation}
\begin{equation}\label{26}
\bm s_k = a_k S(U).
\end{equation}
Here $a_k, b_{kj}$ are parameters satisfying
\begin{equation}\label{27}
 \sum_{k=1}^N a_k = 1, \quad  \sum_{k=1}^N b_{kj} =0, \quad
 \sum_{k=1}^N v_{kj} a_k = 0, \quad  \sum_{k=1}^N v_{kj} b_{ki} = \delta_{ij}, \quad \forall i,j=1,2,...,d.
\end{equation}

Note that the midpoint rule is employed for the source term in the LBM \eqref{22}.
For convenience of implementation, we rewrite it as
%
\begin{align*}
& \bm{F}_k\bigl(\bm{x}_{i} ,t_{n+1}\bigr)
=
 \bm{f}_k\bigl(\bm{x}_{i} - \bm{e}_k \delta_x, t_n\bigr)
-\omega
\left(
\bm{f}_k-\bm{f}_k^{(eq)}
\right) \bigl(\bm{x}_{i} - \bm{e}_k \delta_x,t_n\bigr)\\
& \hspace{2.5cm}
 +  \frac{1}{2} \delta_t \bm{s}_k\bigl(\bm{x}_{i} - \bm{e}_k \delta_x,t_n\bigr), \\
&
\bm{f}_k\bigl(\bm{x}_{i},t_{n+1}\bigr)
= \bm{F}_k\bigl(\bm{x}_{i},t_{n+1}\bigr)
  + \frac{1}{2} \delta_t \bm{s}_k\bigl(\bm{x}_{i},t_{n+1}\bigr).
\end{align*}
Then $U( \bm{x}_{i},t_{n+1} )$ can be obtained by solving
\begin{equation}\label{28}
U(\bm{x}_{i},t_{n+1})
=
\sum_{k=1}^N \bm{F}_k(\bm{x}_{i},t_{n+1}) + \frac{1}{2} \delta_t S(U(\bm{x}_{i},t_{n+1})).
\end{equation}

The key components of the above LBM for hyperbolic systems are the relaxation parameter $\omega$ and the equilibrium distribution $\bm{f}_k^{(eq)}$. As demonstrated in, e.g., \cite{WissocqLiuAbgrall2025,AnandanRao2024}, the scheme \eqref{22} achieves second-order accuracy when $\omega = 2$. However, for problems involving discontinuities, this setting leads to severe spurious oscillations. While setting $\omega = 1$ can significantly reduce these oscillations, it reduces the scheme to first-order accuracy. 
To retain second-order accuracy in smooth regions while reducing spurious oscillations near discontinuities, in the following we propose an adaptive relaxation parameter $\omega$. Additionally, a global Lax-Friedrichs-type equilibrium is designed to further improve the robustness for discontinuous problems.



\subsection{Adaptive relaxation parameter}

The construction of our adaptive relaxation parameter is based on the observation that the LBM is still second-order accurate for $\omega = 2 - O(\delta_x)$ (see the consistency analysis in the next section). It employs local smoothness indicators constructed from local characteristic projections.
\subsubsection{One-dimensional case}

For the 1D system \eqref{21} with $d=1$, let $A(U)=\partial F/\partial U$ be the Jacobian with eigen-decomposition $A = R \Lambda L$, where $L = R^{-1}$ and $\Lambda=\mathrm{diag}(\lambda^{(1)},\dots,\lambda^{(m)})$. Set $U_i:=U(x_i, t_n)$ and denote $L_i$ as the matrix $L$ evaluated at $U_i$.

At each lattice node $x_i$, define backward and forward differences in the same local characteristic basis:
\[ \Delta_-W_i := L_i\left(U_i-U_{i-1}\right), \qquad \Delta_+W_i := L_i\left(U_{i+1}-U_i\right). \]
For the $r$-th characteristic field, we define the local smoothness indicator as
\begin{equation} \label{eq:smoothness_indicator_1d}
	\theta_i^{(r)} = \frac{ \left| \left(\Delta_+W_i\right)^{(r)} - \left(\Delta_-W_i\right)^{(r)} \right| }{ \left| \left(\Delta_-W_i\right)^{(r)} \right| + \left| \left(\Delta_+W_i\right)^{(r)} \right| + \delta x } + \left| \left(\Delta_-W_i\right)^{(r)} \right| + \left| \left(\Delta_+W_i\right)^{(r)} \right|, 
\end{equation}
where $\left(\Delta_+W_i\right)^{(r)}$ stands for the $r$-th component of $\Delta_+W_i$ and $ r=1,2,\ldots,m$.
Summing over all characteristics gives the total indicator
\begin{equation}\label{eq:nu_1d}
	\nu_i = \sum_{r=1}^{m} \theta_i^{(r)},
\end{equation}
which is then smoothed over three neighboring points:
\begin{equation}\label{eq:nu_smoothed_1d}
	\tilde{\nu}_i = \nu_{i-1} + \nu_i + \nu_{i+1}.
\end{equation}

With the above smoothness indicator $\tilde{\nu}_i$, we adaptively choose the relaxation parameter $\omega_i$ as
\begin{equation}\label{eq:Omega_omega_1d}
	\omega_i = 1 + \frac{1}{1+\tilde{\nu}_i}.
\end{equation}
It is clear that $\omega_i \in (1,2]$ since $\tilde{\nu}_i \geq 0$. Additionally, for smooth solutions,
the numerator of the fractional term in $\theta_i^{(r)}$ is
$O(\delta_x^2)$, whereas its denominator is $O(\delta_x)$. Thus,
$\theta_i^{(r)}=O(\delta_x)$ and consequently
$\tilde{\nu}_i=O(\delta_x)$. This yields
$\omega_i=2-O(\delta_x)$, leading to the second-order accuracy of the
method (see the consistency analysis in the next section). Near
discontinuities, we have $\tilde{\nu}_i\gg1$ and $\omega_i\approx1$,
which introduces additional numerical dissipation to suppress spurious
oscillations and enhance the robustness of the method.

\subsubsection{Multi-dimensional case}


For the general $d$-dimensional system \eqref{21}, let
$A_j(U)=\partial F_j/\partial U$ be the Jacobian in direction $x_j$ and assume it has 
the eigen-decomposition $A_j(U)=R_{j}\Lambda_{j}L_{j}$, where $L_{j}=R_{j}^{-1}$ and $\Lambda_j=\mathrm{diag}(\lambda_j^{(1)},\dots,\lambda_j^{(m)})$. Set $U_i:=U(\bm x_i, t_n)$ and denote $L_{i,j}$ as the matrix $L_j$ evaluated at $U_i$.

Note that the LBM uses a uniform Cartesian grid with the same mesh spacing $ \delta_x $ in all coordinate directions. For each direction $x_j$, define the local backward and forward
characteristic projections at the grid point $\boldsymbol{x}_i$ as
\[
\Delta_{j,-}W_i
:=
L_{i,j}
\left(
U_i-U_{i-\boldsymbol{n}_j}
\right),
\qquad
\Delta_{j,+}W_i
:=
L_{i,j}
\left(
U_{i+\boldsymbol{n}_j}-U_i
\right),
\]
where $\boldsymbol{n}_j$ is the unit multi-index in direction $j$ and $U_{i \pm \boldsymbol{n}_j}:=U(\bm x_{i \pm \boldsymbol{n}_j}, t_n)$.
Similar to the 1D case, define the directional indicator as
\begin{equation}\label{213}
	\theta^{(r,j)}( \bm x_i )
	=
	\frac{
		\left|
		\left(\Delta_{j,+}W_i\right)^{(r)}
		-
		\left(\Delta_{j,-}W_i\right)^{(r)}
		\right|
	}{
		\left|
		\left(\Delta_{j,-}W_i\right)^{(r)}
		\right|
		+
		\left|
		\left(\Delta_{j,+}W_i\right)^{(r)}
		\right|
		+
		\delta x
	}
	+
	\left|
	\left(\Delta_{j,-}W_i\right)^{(r)}
	\right|
	+
	\left|
	\left(\Delta_{j,+}W_i\right)^{(r)}
	\right|.
\end{equation}

Summing \eqref{213} over all characteristics and directions gives
\begin{equation}\label{eq:nu_multi}
	\nu( \bm x_i ) = \sum_{j=1}^{d}\sum_{r=1}^{m} \theta^{(r,j)}( \bm x_i ),
\end{equation}
which is smoothed over the $3^d$ neighboring nodes:
\begin{equation}\label{eq:nu_smoothed_multi}
	\tilde{\nu}( \bm x_i ) = \sum_{\bm{k}\in\{-1,0,1\}^d} \nu( \bm x_{i} + \bm{k} \delta_x).
\end{equation}
The relaxation parameter is then given by
\begin{equation}\label{eq:Omega_omega_multi}
	\omega(\bm x_{i}) = 1 + \frac{1}{ 1+\tilde{\nu}(\bm x_{i}) }.
\end{equation}
Similar to the 1D case, it is easy to see that $\omega(\bm x_{i}) = 2 - O(\delta_x)$ in smooth regions and $\omega(\bm x_{i}) \approx 1$  near discontinuities.


\subsection{Lax-Friedrichs-type equilibrium}

To further improve the robustness for discontinuous problems, we introduce a Lax-Friedrichs-type equilibrium.
Specifically, for the 1D case we define
\begin{equation}\label{eq:alpha_1d}
\bar \alpha = \max_{x_i} \max_{1\le r\le m} |\lambda^{(r)}|,
\end{equation}
where $\lambda^{(r)}$ are the eigenvalues of $A(U_i)$. Then the equilibrium distributions with D1Q3 lattice ($\bm e_0 =0, \bm e_1 = 1, \bm e_2 = -1$) are given by
\begin{subequations}\label{217}
\begin{align}
\bm{f}_0^{(eq)} &= \left(1 - \frac{\alpha}{\lambda}\right)U, \label{eq:eq0_1d} \\
\bm{f}_1^{(eq)} &= \frac{ \alpha U  +  F(U)}{2\lambda}, \label{eq:eqp_1d} \\
\bm{f}_2^{(eq)} &= \frac{\alpha U - F(U)}{2\lambda}, \label{eq:eqm_1d}
\end{align}
\end{subequations}
where $\alpha = \bar \alpha + \epsilon$ with $\epsilon>0$ a small parameter and $\lambda = \delta_x / \delta_t$. Note that $\bm{f}_1^{(eq)}$ and $\bm{f}_2^{(eq)}$ can be written as
\begin{equation*}
\begin{split}
& \bm{f}_1^{(eq)} = \frac{\alpha}{2\lambda} ( U  +  \frac{1}{\alpha}F(U) ) := \frac{\alpha}{2\lambda} F^{+},\\
& \bm{f}_2^{(eq)} = \frac{\alpha}{2\lambda} ( U  -  \frac{1}{\alpha}F(U) ) := \frac{\alpha}{2\lambda} F^{-},
\end{split}
\end{equation*}
where $F^{+}=U  +  \frac{1}{\alpha}F(U) $ and $F^{-}=U  -  \frac{1}{\alpha}F(U) $ are the Lax-Friedrichs fluxes. Thus we refer to the above equilibrium distributions as Lax-Friedrichs-type equilibria.

For the general $d$-dimensional case, let $\alpha_j  = \bar \alpha_j + \epsilon$ with $\epsilon>0$ a small parameter and
\begin{equation}\label{eq:alpha_multi}
\bar \alpha_j = \max_{ \bm x_{i} }  \max_{1\le r\le m} |\lambda^{(r)}( \partial_U F_j )|, \quad j=1,2,...,d.
\end{equation}
Here $\lambda^{(r)}( \partial_U F_j )$ are eigenvalues of $\partial_U F_j$.
For the D$d$Q$(2d+1)$ lattice, we take $\bm e_0=(0,0,...,0)$, $\bm e_j = (0,...,0,1,0,...,0)$ as the unit vector in the $j$-th coordinate direction, and $\bm e_{j+d} = - \bm e_j, j=1,2,...,d$. The equilibrium distributions are then defined by
\begin{subequations}\label{n220}
\begin{align}
\bm{f}_0^{(eq)} &= \left(1 - \frac{\sum_{j=1}^d \alpha_j }{\lambda}\right)U, \label{eq:eq0_multi} \\
\bm{f}_j^{(eq)} &= \frac{\alpha_j U + F_j(U)}{2\lambda},\quad j=1,\dots,d, \label{eq:eqp_multi} \\
\bm{f}_{j+d}^{(eq)} &= \frac{\alpha_j U - F_j(U)}{2\lambda},\quad j=1,\dots,d. \label{eq:eqm_multi}
\end{align}
\end{subequations}

Note that the above equilibrium satisfies conditions in \eqref{27} with $\bm v_j = \lambda \bm e_j$, and thus the consistency is guaranteed. Moreover, under the above Lax-Friedrichs-type equilibrium, the time step is determined by
\begin{equation}\label{eq:dt}
\delta_t = \mathrm{CFL}  \dfrac{\delta_x}{ \sum_{j=1}^d \bar \alpha_j}
=   \mathrm{CFL}  \dfrac{\delta_x}{ \sum_{j=1}^d \max_{ \bm x_{i} }  \max_{1\le r\le m} |\lambda^{(r)}( \partial_U F_j )|} .
\end{equation}
We show in Section \ref{sec4} that the LBM with the above adaptive relaxation parameter and  Lax-Friedrichs-type equilibrium is stable for linear systems under the standard CFL condition $\mathrm{CFL} < 1$.


\section{Consistency analysis}
\setcounter{equation}{0}

In this section, we use the Maxwell iteration \cite{IkenberryTruesdell1956,YongZhaoLuo2016,ZhaoYong2017} to conduct a consistency analysis for the LBM \eqref{22}. 
We show that the present LBM \eqref{22} with the adaptive relaxation parameter $\omega$ is second-order accurate since  $\omega = 2 - O(\delta_x)$.

To begin with, we set $h:=\delta_x, \delta_t = ch:= \frac{1}{\lambda} h$ and expand the left-hand side of \eqref{22}  as
\begin{equation}\label{34}
\begin{split}
 \bm f_k (\bm{x}_\ell + \bm{e}_k h, t_n +  c h)
 &= \sum\limits_{p \ge 0} h^p
       \frac{(c\partial_t + \bm e_k \cdot \nabla)^p}{p!} \bm f_k (\bm{x}_\ell, t_n)\\
 &:= \sum\limits_{p \ge 0} h^p D_{k, p} \bm f_k (\bm{x}_\ell, t_n),
\end{split}
\end{equation}
where $D_{k, p}$ is defined as
$$
D_{k, p}
=
\frac{1}{p!}(c\partial_t + \bm e_k \cdot \nabla)^p.
$$
Denote $\bm{f}_k:=\bm{f}_k(\bm{x}_\ell, t_n)$, then we deduce from the LBM \eqref{22} that
$$
\sum\limits_{ 1 \le p \le 2} h^p D_{k, p} \bm f_k = -\omega
(
\bm{f}_k-\bm{f}_k^{(eq)}
)
+ ch \bm{s}_k + \frac{1}{2} c h^2 D_{k, 1} \bm{s}_k + O(h^3).
$$
Assume $\omega \neq 0$ and define $\tau := 1/ \omega$. The above equation can be written as
$$
\bm{f}_k = \bm{f}_k^{(eq)} - h \tau D_{k, 1}\bm f_k - h^2 \tau D_{k, 2}\bm f_k
+  \tau ch \bm{s}_k + \frac{1}{2} \tau c h^2  D_{k, 1}  \bm{s}_k + O(h^3).
$$
Using the Maxwell iteration \cite{IkenberryTruesdell1956,YongZhaoLuo2016,ZhaoYong2017} for the above equation, \ie, substituting the above expression of $\bm{f}_k$ into its right-hand side, we have
\begin{equation}\label{324}
\begin{split}
\bm{f}_k
& = \bm{f}_k^{(eq)} - h \tau D_{k, 1}\bm{f}_k^{(eq)}
+ h^2(  \tau^2 D^2_{k, 1} - \tau D_{k, 2}  ) \bm{f}_k^{(eq)}\\
& \hspace{5mm}-  \tau^2 ch^2 D_{k, 1} \bm{s}_k
+ \tau ch \bm{s}_k + \frac{1}{2} \tau c h^2 D_{k, 1} \bm{s}_k + O(h^3) \\
& = \bm{f}_k^{(eq)} - h \tau D_{k, 1}\bm{f}_k^{(eq)}
+ h^2 \tau (  \tau -\frac{1}{2} )  D^2_{k, 1} \bm{f}_k^{(eq)}\\
& \hspace{5mm} + \tau ch \bm{s}_k -  \tau(\tau - \frac{1}{2}) ch^2 D_{k, 1} \bm{s}_k + O(h^3).
\end{split}
\end{equation}

Now we sum the expansion \eqref{324} over $k$ to obtain
\begin{equation}\label{326}
\begin{split}
\sum_k \bm{f}_k
=
&U - \tau ch [ \partial_t U
+ \sum_{j=1}^d \partial_{x_j} F_j(U) ]
+  \tau (  \tau -\frac{1}{2} )  h^2 \sum_k D^2_{k, 1} \bm{f}_k^{(eq)} \\
& + \tau ch  S   - \tau(\tau - \frac{1}{2}) ch^2 \sum_k D_{k, 1} \bm{s}_k  + O(h^3), \\
=& U - \tau ch [ \partial_t U
+ \sum_{j=1}^d \partial_{x_j} F_j(U) ]
+ \tau (  \tau -\frac{1}{2} ) h^2  \sum_k D^2_{k, 1} \bm{f}_k^{(eq)} \\
& + \tau ch S  - \tau(\tau - \frac{1}{2}) c^2 h^2 \partial_t S + O(h^3),
\end{split}
\end{equation}
where
$$
\sum_k D_{k, 1} \bm{s}_k
= \sum_k (c\partial_t + \bm e_k \cdot \nabla) \bm{s}_k
=
c \sum_k \partial_t  \bm{s}_k = c \partial_t S
$$
has been used for the second equality. With \eqref{326} and $U=\sum_k \bm{f}_k$, we have
\begin{equation}\label{331}
\begin{split}
 \partial_t U
+ \sum_{j=1}^d \partial_{x_j} F_j(U)
=  S
+  (  \tau -\frac{1}{2} ) \frac{1}{c} h  \sum_k D^2_{k, 1} \bm{f}_k^{(eq)}
- (\tau - \frac{1}{2})ch \partial_t S
+  O(h^2).
\end{split}
\end{equation}
This equation means that the error is $O(h)$ and thus the method is generally first-order accurate. If $\omega = 2 - O(\delta_x) = 2 - O(h)$, then $\tau = \frac{1}{2} + O(h)$, and thereby
$$
(  \tau -\frac{1}{2} ) \frac{1}{c} h  \sum_k D^2_{k, 1} \bm{f}_k^{(eq)} = O(h^2),
\quad
(\tau - \frac{1}{2})ch \partial_t S = O(h^2).
$$
Thus \eqref{331} further becomes
\begin{equation}\label{334}
\begin{split}
 \partial_t U
+ \sum_{j=1}^d \partial_{x_j} F_j(U)
=  S
+  O(h^2),
\end{split}
\end{equation}
which indicates that the LBM has second-order accuracy.

The above analysis shows that under our adaptive relaxation parameter \eqref{eq:Omega_omega_multi} (or \eqref{eq:Omega_omega_1d} for 1D case), the LBM is second-order accurate for smooth solutions.


\section{Weighted \texorpdfstring{$L^2$}{L2}-stability}\label{sec4}
\setcounter{equation}{0}


This section is devoted to proving the weighted $L^2$-stability of the LBM \eqref{22} for linear systems, \ie,  $F_j(U)=A_jU$ with $A_j$ being constant matrices. This is an extension of the classical framework of the weighted $L^2$-stability for the standard  LBM \cite{BandaYongKlar2006,JunkYong2009}. 
In this analysis, the source term is ignored while it can be  treated in the convergence analysis as in \cite{JunkYang2009}. 


Since the system  \eqref{21} is hyperbolic, there exists a symmetric positive-definite matrix $A_0$ such that $A_0 A_j, j=1,2,...,d$ are all symmetric. Additionally, $A_0$ is constant for linear systems.
Then we define
\begin{equation}\label{n41}
\begin{split}
\hat A_0  := 
& \diag( A_0( a_1 I_m + \sum_{j=1}^d b_{1j} A_j )^{-1}, A_0( a_2 I_m + \sum_{j=1}^d b_{2j} A_j )^{-1},\\
 &\hspace{8mm} ..., A_0( a_N I_m + \sum_{j=1}^d b_{Nj} A_j )^{-1} ),
\end{split}
\end{equation}
which is positive definite if so is $A_0(a_k I_m + \sum_{j=1}^d b_{kj} A_j)^{-1}$ for all $k$. This is true if $A_0(a_k I_m + \sum_{j=1}^d b_{kj} A_j), j=1,2,...,d$ are positive definite. The symmetry of $A_0(a_k I_m + \sum_{j=1}^d b_{kj} A_j)$ is obvious and its positive definiteness is equivalent to
\begin{equation}\label{111}
\mbox{the eigenvalues of~} a_k I_m + \sum_{j=1}^d b_{kj} A_j \mbox{~are all positive}, \quad \forall k.
\end{equation}

With the above positive definite matrix $\hat A_0$, we define $ \bm Q_k := -\omega (  \bm f_k - \bm f_k^{(eq)} )$ and show that the collision term $\bm Q(\bm f):=(\bm Q_1^T, \bm Q_2^T, ..., \bm Q_N^T)^T$ admits the following stability structure, which is an extension of that for the standard LBM in \cite{BandaYongKlar2006}.

\begin{thm}[Stability structure]\label{thm21}
Denote $J=\frac{\partial \bm Q(\bm{f})}{\partial \bm f}$.
Under the condition \eqref{111}, there exists an invertible matrix $P \in \mathbb{R}^{Nm \times Nm}$
such that $P^{T} P = \hat A_0$ and
\begin{equation}\label{eqn:StabiliStructure}
  P J P^{-1} = \diag (\lambda_1, \lambda_2, \ldots, \lambda_{Nm} ),
\end{equation}
where each eigenvalue $\lambda_i$ is either $0$ or $-\omega$.

\end{thm}

\begin{proof}
Note that
$
\frac{\partial \bm f_k^{eq}(\bm{f})}{\partial \bm f_i}
 =  a_k I_m + \sum_{j=1}^d b_{kj} A_j \equiv E_k
$ is independent of $i$, thus $\frac{\partial \bm f^{(eq)}(\bm{f})}{\partial \bm f}$ is a block matrix with the submatrix in each row the same, \ie,
\begin{equation}\label{eq:Eblock}
E:=\frac{\partial \bm f^{(eq)}(\bm{f})}{\partial \bm f}
=
\left(
\begin{array}{llll}
E_1 & E_1 & \cdots & E_1 \\
E_2 & E_2 & \cdots & E_2 \\
\vdots & \vdots & \vdots & \vdots\\
E_N & E_N & \cdots & E_N
\end{array}
\right)
=
\left(
\begin{array}{l}
E_1  \\
E_2  \\
\vdots \\
E_N
\end{array}
\right)
(I_m, I_m, ..., I_m).
\end{equation}
Then we compute
\begin{equation*}
\hat A_0 E =
\left(
\begin{array}{llll}
A_0 & A_0 & \cdots & A_0 \\
A_0 & A_0 & \cdots & A_0 \\
\vdots & \vdots & \vdots & \vdots\\
A_0 & A_0 & \cdots & A_0
\end{array}
\right),
\end{equation*}
which is symmetric. Thus $\hat A_0 J$ is symmetric. Since $\hat A_0$ is symmetric and positive definite, there exists an invertible matrix $P$ such that $P^{T} P = \hat A_0$ and $P J P^{-1}$ is diagonal.

Next we analyze the eigenvalues of $E$. According to the decomposition in \eqref{eq:Eblock}, the rank of $E$ is at most $m$.
On the other hand, note that $\sum_k E_k = I_m$, \ie, the sum of submatrices in each column is $I_m$. Thus the matrix $E$ has an eigenvalue 1 with multiplicity being $m$.
Then the eigenvalues of $E$ are 0 and 1, and those of $J=-\omega(I_{Nm}-E)$ are 0 and $-\omega$. This completes the proof.

\end{proof}

With the matrix $P$ in the above stability structure, we define the following weighted $L^2$-norm:
\begin{equation*}
\begin{split}
\Vert \mathbf{f} ( t_{n} )  \Vert_{P}^2
&:=
\sum_{\bm x_{\ell} \in \Omega} (P\bm f( \bm x_{\ell}, t_n ))^T  P\bm f( \bm x_{\ell}, t_n )\\
&=
\sum_{\bm x_{\ell} \in \Omega} \bm f^T( \bm x_{\ell}, t_n ) \hat A_0 \bm f( \bm x_{\ell}, t_n )\\
&=
\sum_{\bm x_{\ell} \in \Omega} \sum_{k=1}^N \bm f_k^T( \bm x_{\ell}, t_n ) A_0( a_k I_m + \sum_{j=1}^d b_{kj} A_j )^{-1} \bm f_k( \bm x_{\ell}, t_n ).
\end{split}
\end{equation*}
Then we have the following stability result, the proof of which is the same as that in \cite{JunkYong2009} for the standard LBM.
%

\begin{thm}\label{thm22}
Assume that condition \eqref{111} holds. If the relaxation parameter satisfies $0\leq \omega \leq 2$
for all grid points and time levels, then the solution of the LBM \eqref{22} without source term satisfies
\begin{equation}\label{45}
	\| \bm{f}(t_{n+1})\|_P \leq \| \bm{f}(t_n)\|_P
\end{equation}
under periodic boundary conditions.
\end{thm}

Applying the above general result to the proposed LBM with adaptive relaxation parameter and Lax-Friedrichs-type equilibria, we have
\begin{thm}\label{thm43}
If the $\mathrm{CFL}$ number defined in \eqref{eq:dt} satisfies $\mathrm{CFL} < 1$ and $\alpha_j=\bar{\alpha}_j+\epsilon, j=1,\ldots,d,$ with $0 < \epsilon < \frac{1}{d}( \frac{1}{\mathrm{CFL}} - 1) \sum_{j=1}^d \bar \alpha_j$, then the LBM \eqref{22} with the adaptive relaxation parameter \eqref{eq:Omega_omega_multi} and Lax-Friedrichs-type equilibria \eqref{n220} possesses the weighted $L^2$-stability  \eqref{45} for periodic problems.
\end{thm}
\begin{proof}
Note that our adaptive relaxation parameter \eqref{eq:Omega_omega_multi} automatically satisfies the condition $\omega \in [0,2]$.
To verify the condition \eqref{111} for the Lax-Friedrichs-type equilibria \eqref{n220}, it is sufficient to require
$$
\alpha_j > \bar \alpha_j, \quad  \sum_{j=1}^d \alpha_j < \lambda.
$$
Recall that $\alpha_j = \bar \alpha_j + \epsilon$ with $\epsilon >0$ and thus $\alpha_j > \bar \alpha_j$ is true.
On the other hand, according to \eqref{eq:dt}, we have $\lambda = \frac{\delta_x}{\delta_t} = \frac{\sum_{j=1}^d \bar \alpha_j}{\mathrm{CFL}}$. Then $\sum_{j=1}^d \alpha_j  < \lambda$ becomes
$$
\sum_{j=1}^d \bar \alpha_j + d \epsilon < \frac{1}{\mathrm{CFL}} \sum_{j=1}^d \bar \alpha_j,
$$
which holds if $\epsilon < \frac{1}{d}( \frac{1}{\mathrm{CFL}} - 1) \sum_{j=1}^d \bar \alpha_j$. The proof is complete.

\end{proof}

\begin{rem}

The stability condition $\mathrm{CFL}<1$ obtained here is standard under the present weighted \(L^2\)-stability framework.
In comparison, for the LBM with D2Q5 lattice proposed in \cite{WissocqLiuAbgrall2025}, the convex-preserving property of the first-order scheme requires the more restrictive condition $ \mathrm{CFL}\leq \frac{1}{4} $ (see Proposition 2.7 of \cite{WissocqLiuAbgrall2025}). Thus, under the present weighted \(L^2\)-stability criterion, the proposed Lax--Friedrichs-type equilibrium admits a less restrictive CFL condition.

\end{rem}


\begin{rem}
With the weighted \(L^2\)-stability in Theorem \ref{thm22}, convergence can be obtained by combining the stability estimate with the consistency analysis, following the standard argument in \cite{JunkYang2009}.
However, for nonlinear hyperbolic systems, proving weighted \(L^2\)-stability, and hence convergence, remains challenging.
A similar restriction
applies to the \(L^\infty\)-stability analysis in \cite{Graille2014}, which is
established only for scalar equations. We also note that the recent works
\cite{GuillonHelieHelluy2024,BellottiHelluyNavoret2025} developed
entropy-stable LBMs for nonlinear hyperbolic systems. These approaches,
however, require suitable convex kinetic entropies, whose construction remains
unresolved for several systems, \eg, the compressible Euler
equations.
\end{rem}

\section{Numerical experiments}
\setcounter{equation}{0}
In this section, we validate the accuracy and robustness of the present LBM for several typical hyperbolic equations, including the linear advection equation, Euler equations, shallow-water equations, and reactive Euler equations. 
In all our examples, we set $\epsilon=10^{-10}$, which satisfies the condition in Theorem \ref{thm43}. 


\subsection{Linear advection equations}
We first use linear advection to isolate the effects of smoothness and discontinuities. Periodic boundary conditions are imposed in all tests.

\noindent\textbf{Example 5.1 (Smooth advection).}
Consider the 1D linear advection equation
\[
u_t+u_x=0,\qquad x\in[0,1]
\]
with initial data $u(x,0)=\sin(2\pi x)$. The exact solution is $u(x,t)=\sin(2\pi(x-t))$, and we set $\mathrm{CFL}=0.9$. 
As shown in Table~\ref{tab:linear_no_source}, the error at $t=1$ decreases at the expected second-order rate. The smoothness indicator remains small, so $\omega$ stays close to $2$ without degrading the accuracy of the underlying scheme.
\begin{table}[htbp]
\centering
\begin{tabular}{cccc}
\toprule
$N_x$ & $\delta_x$ & $L^2$-error & Order\\
\midrule
400&$2.500\times10^{-3}$&$2.5507\times10^{-4}$&$--$\\
800&$1.250\times10^{-3}$&$6.3095\times10^{-5}$&2.0153\\
1600&$6.250\times10^{-4}$&$1.6364\times10^{-5}$&1.9470\\
3200&$3.125\times10^{-4}$&$4.1924\times10^{-6}$&1.9647\\
6400&$1.563\times10^{-4}$&$9.7916\times10^{-7}$&2.0982\\
\bottomrule
\end{tabular}
\caption{Example 5.1: Relative $L^2$-errors and convergence orders for smooth linear advection at $t=1$.}
\label{tab:linear_no_source}
\end{table}

\noindent\textbf{Example 5.2 (1D square wave).}
In this example, we solve $u_t+u_x=0$ on $[0,1]$ with initial data
\[
u(x,0)=\begin{cases}1,&x\in[0.25,0.75],\\0,&\text{otherwise}.
\end{cases}
\]
With periodic boundary conditions, the exact solution is the translated initial profile, \ie, $u(x,t)=u_0(x-t)$. The computation is advanced to $t=1$ with $\mathrm{CFL}=0.9$. Figure~\ref{fig:fangbo1D} compares the numerical and exact solutions on successively refined meshes; the right panel shows $\omega$ for $N_x=800$. The two discontinuities are captured without visible spurious oscillations, and the numerical transition layers narrow under refinement. The relaxation parameter remains close to $2$ in the constant regions and decreases only near the jumps. Thus, the present LBM introduces additional dissipation locally without unnecessarily smearing the solution away from the discontinuities.
\begin{figure}[htbp]
\centering
\begin{minipage}{0.45\textwidth}\centering\includegraphics[width=\linewidth]{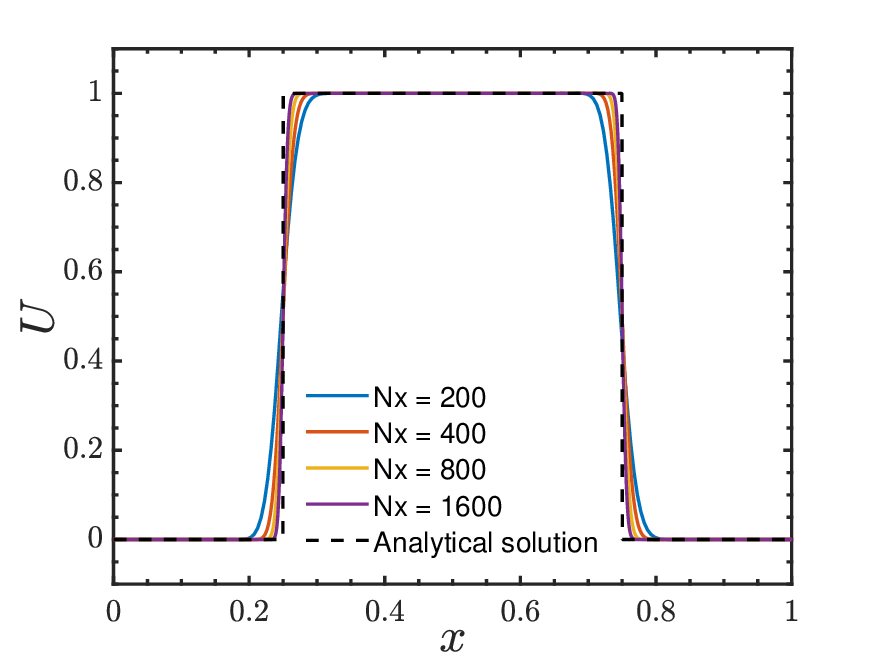}\end{minipage}\quad
\begin{minipage}{0.45\textwidth}\centering\includegraphics[width=\linewidth]{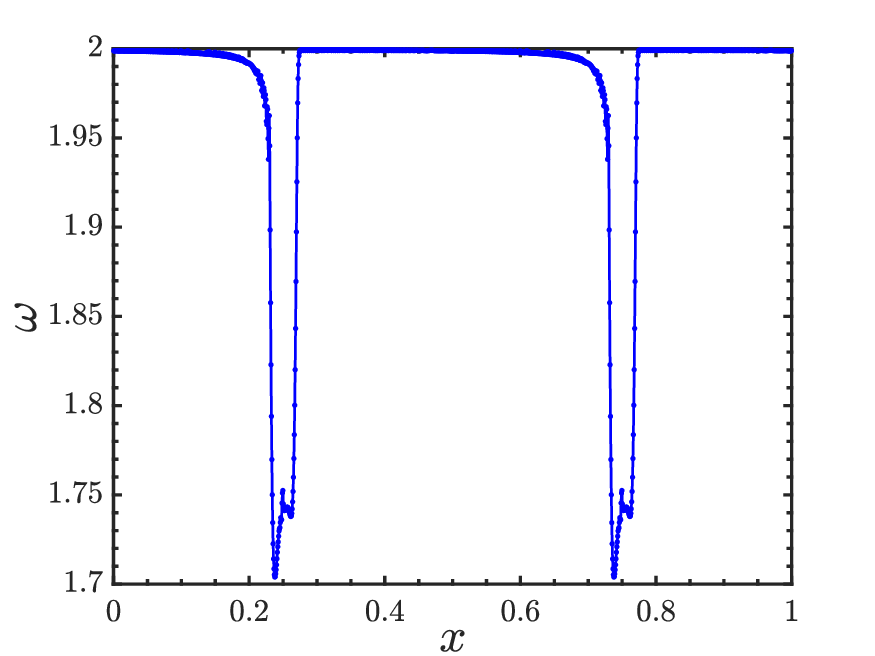}\end{minipage}
\caption{Example 5.2: Left: Numerical and exact square-wave profiles under mesh refinement; Right: distribution of the adaptive relaxation parameter $\omega$ with $N_x=800$ at $t=1$.}
\label{fig:fangbo1D}
\end{figure}

\subsection{Euler equations}
We next compute the Euler equations. We first verify accuracy with smooth solutions and then simulate shock-dominated flows. Periodic and transmissive boundary conditions are used for the smooth and shock-tube problems, respectively.

\noindent\textbf{Example 5.3 (1D entropy wave).}
On $x\in[0,1]$, let
\[
\rho(x,0)=1+0.2\sin(2\pi x),\qquad u(x,0)=1,\qquad p(x,0)=1.
\]
The exact density is $\rho(x,t)=1+0.2\sin(2\pi(x-t))$, while velocity and pressure remain constant. The computation is carried out to $t=1$ with $\mathrm{CFL}=0.6$. Table~\ref{tab:Ola_1D} shows second-order convergence of the density error under mesh refinement. The characteristic-projection-based sensor keeps $\omega$ close to $2$, confirming that it introduces no excessive dissipation for smooth solutions of the Euler equations.
\begin{table}[htbp]
\centering
\begin{tabular}{cccc}
\toprule
$N_x$&$\delta_x$&$L^2$-error&Order\\
\midrule
400&$2.500\times10^{-3}$&$2.0196\times10^{-4}$&$--$\\
800&$1.250\times10^{-3}$&$5.0515\times10^{-5}$&1.9992\\
1600&$6.250\times10^{-4}$&$1.2841\times10^{-5}$&1.9759\\
3200&$3.125\times10^{-4}$&$3.3697\times10^{-6}$&1.9736\\
6400&$1.563\times10^{-4}$&$8.0990\times10^{-7}$&2.0133\\
\bottomrule
\end{tabular}
\caption{Example 5.3: Relative $L^2$-errors and convergence orders for 1D entropy-wave advection of the Euler equations at $t=1$.}
\label{tab:Ola_1D}
\end{table}

\noindent\textbf{Example 5.4 (2D diagonal entropy wave).}
On $[0,1]^2$, the initial state is
\[
\rho(x,y,0)=1+0.2\sin(2\pi(x+y)),\qquad u(x,y,0)=v(x,y,0)=p(x,y,0)=1.
\]
The exact density is $\rho(x,y,t)=1+0.2\sin(2\pi(x+y-2t))$, while velocity and pressure remain constant. The computation is performed to $t=0.3$ with $\mathrm{CFL}=0.3$. As shown in Table~\ref{tab:O_2D_nonlinear}, the convergence rates are pre-asymptotic on the coarser grids, where the diagonal wave is not yet sufficiently resolved and the adaptive indicator introduces some localized dissipation. Under mesh refinement, the observed rate increases to $1.96$, indicating recovery of the second-order convergence for smooth solutions of 2D Euler equations.
\begin{table}[htbp]
\centering
\begin{tabular}{cccc}
\toprule
$N_x$&$\delta_x$&$L^2$-error&Order\\
\midrule
40&$2.500\times10^{-2}$&$2.3553\times10^{-2}$&$--$\\
80&$1.250\times10^{-2}$&$8.2925\times10^{-3}$&1.5060\\
160&$6.250\times10^{-3}$&$2.5326\times10^{-3}$&1.7112\\
320&$3.125\times10^{-3}$&$7.1705\times10^{-4}$&1.8205\\
640&$1.562\times10^{-3}$&$1.9068\times10^{-4}$&1.9109\\
\bottomrule
\end{tabular}
\caption{Example 5.4: Relative $L^2$-errors and convergence orders for 2D diagonal entropy-wave advection of the Euler equations at $t=0.3$.}
\label{tab:O_2D_nonlinear}
\end{table}

We next assess shock capturing and the resolution of small-scale structures using the classical Lax, Sod, and Shu--Osher problems \cite{Lax1954,Sod1978,ShuOsher1989}. These tests contain rarefaction waves, shocks, and contact discontinuities; the Shu--Osher problem additionally involves the interaction of a shock with high-frequency density fluctuations. Transmissive boundary conditions and $\mathrm{CFL}=0.3$ are used throughout. Tables~\ref{tab:shock_config} 
summarize the computational configurations and initial data. Exact Riemann solutions serve as references for the Lax and Sod problems. For the Shu--Osher problem, the reference solution is computed on the same uniform mesh using a fifth-order WENO scheme with global Lax--Friedrichs flux splitting and a third-order TVD Runge--Kutta time integration.
\begin{table}[htbp]
\centering
\begin{tabular}{lcccc}\toprule
Problem&Domain&Final time&Left state $(\rho,u,p)$&Right state $(\rho,u,p)$\\\midrule
Sod&$[0,1]$&0.2&$(1,0,1)$&$(0.125,0,0.1)$\\
Lax&$[0,1]$&0.14&$(0.445,0.698,3.528)$&$(0.5,0,0.571)$\\
Shu--Osher&$[-5,5]$&1.8&$(3.857,2.629,10.333)$&$(1+0.2\sin(5x),0,1)$\\\bottomrule
\end{tabular}
\caption{Examples 5.5--5.7: Computational configurations and initial data for the shock tube problems.}
\label{tab:shock_config}
\end{table}


\noindent\textbf{Example 5.5 (Lax shock tube).}
The Lax problem \cite{Lax1954} contains a rarefaction wave, a contact discontinuity, and a strong shock. Figure~\ref{fig:lax} shows close agreement between the present LBM and the exact solution. Both the contact discontinuity and the shock are captured without visible spurious oscillations. The relaxation parameter decreases primarily near steep gradients and remains close to $2$ elsewhere, thereby retaining low dissipation in the smooth portions of the flow.
\begin{figure}[htbp]\centering
\begin{minipage}{0.44\textwidth}\includegraphics[width=\linewidth]{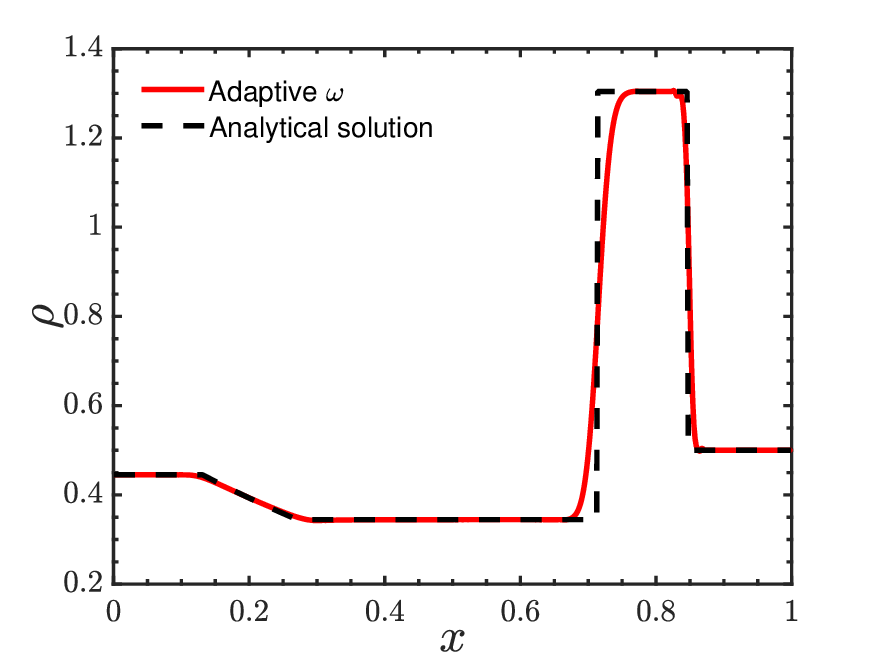}\end{minipage}\quad
\begin{minipage}{0.44\textwidth}\includegraphics[width=\linewidth]{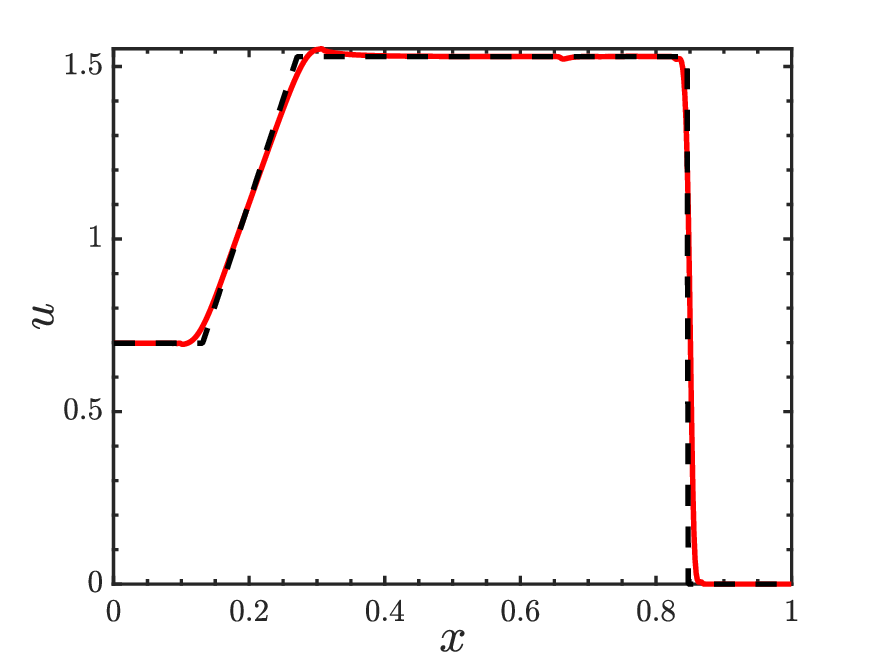}\end{minipage}\\
\begin{minipage}{0.44\textwidth}\includegraphics[width=\linewidth]{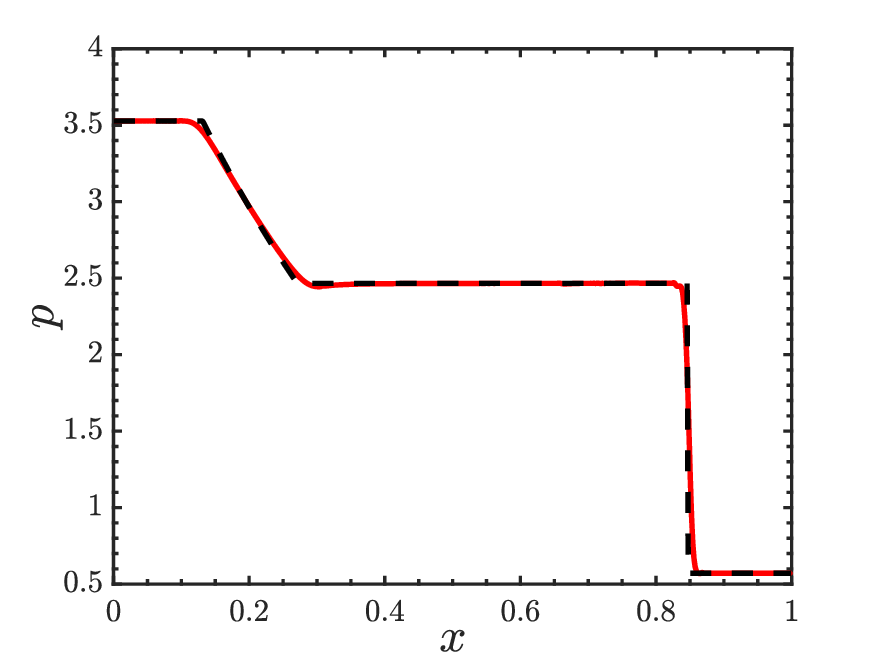}\end{minipage}\quad
\begin{minipage}{0.44\textwidth}\includegraphics[width=\linewidth]{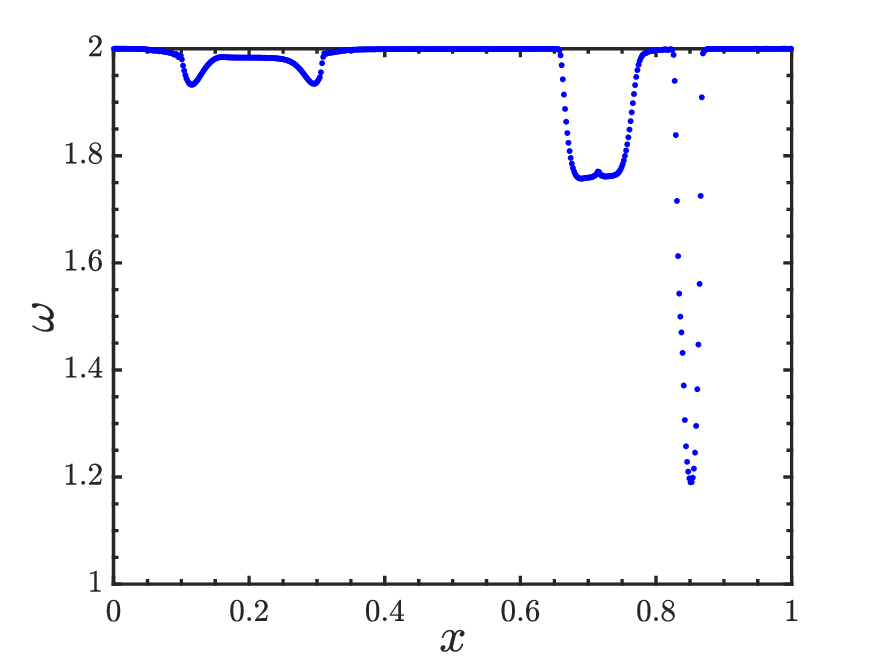}\end{minipage}
\caption{Example 5.5: Lax shock-tube problem at $t=0.14$ with $N_x=600$: density, velocity, pressure, and adaptive relaxation parameter.}\label{fig:lax}
\end{figure}

\noindent\textbf{Example 5.6 (Sod shock tube).}
For the Sod problem \cite{Sod1978}, we compare the present LBM with the fixed-$\omega=1$ scheme, a reproduced relaxed local maximum-principle (RLMP) LBM \cite{WissocqLiuAbgrall2025}, and the exact solution. The global profiles in Figure~\ref{fig:sod} show that the fixed-$\omega=1$ scheme is visibly more dissipative, whereas the present LBM and RLMP method retain sharper wave structures. In the enlarged views in Figure~\ref{fig:sod_duibi}, the present LBM follows the exact profiles smoothly, while the reproduced RLMP solution exhibits small, localized pointwise deviations. Table~\ref{tab:sod_tv_cpu} reports the local total variation over the domain $I=[0.15,0.35]$ and the measured MATLAB CPU time under the same computational setting. For this test and implementation, the present LBM gives smaller local total variations in density, velocity, and pressure and requires less CPU time, indicating smoother local profiles at a lower computational cost.
\begin{figure}[htbp]\centering
\begin{minipage}{0.44\textwidth}\includegraphics[width=\linewidth]{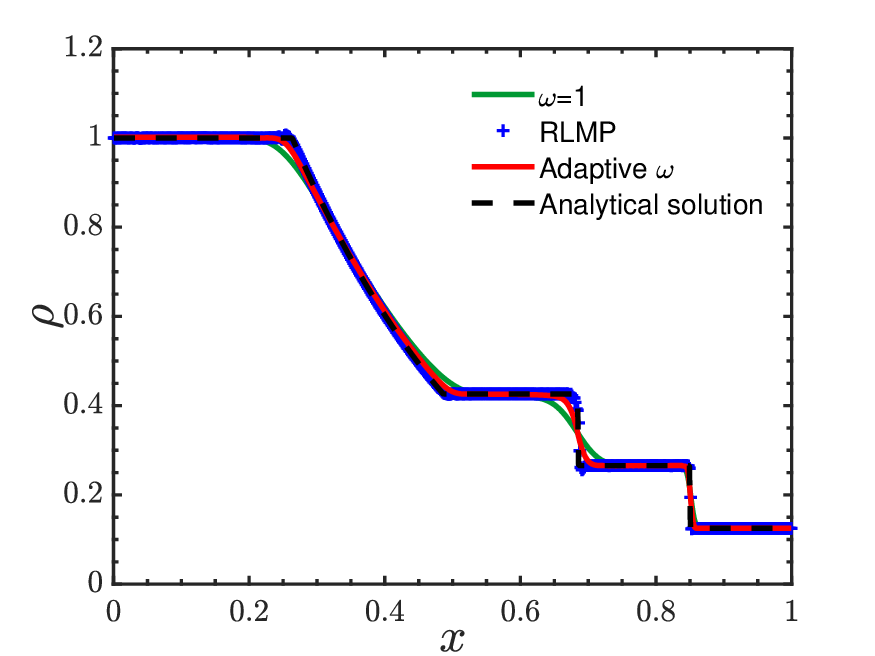}\end{minipage}\quad
\begin{minipage}{0.44\textwidth}\includegraphics[width=\linewidth]{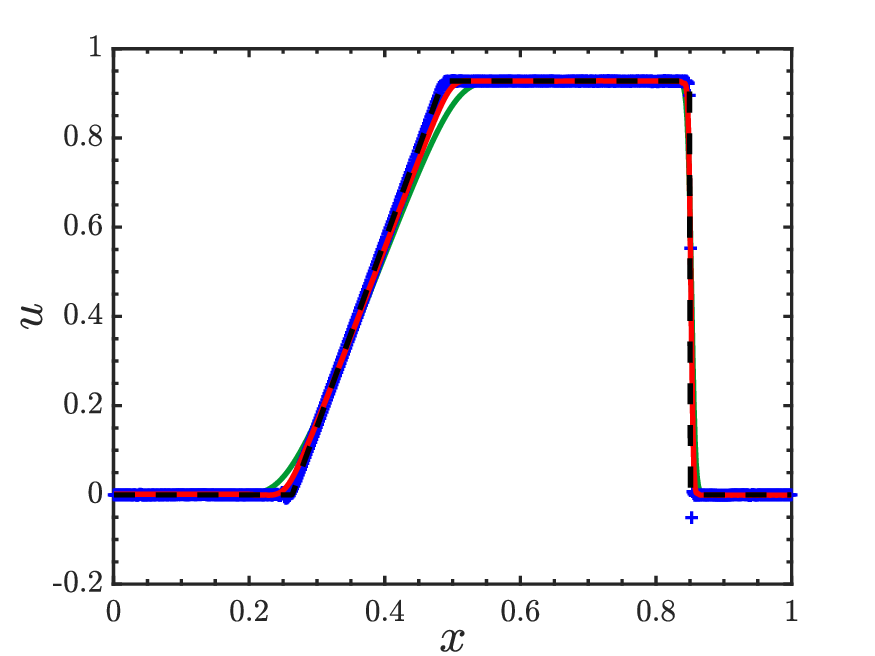}\end{minipage}\\
\begin{minipage}{0.44\textwidth}\includegraphics[width=\linewidth]{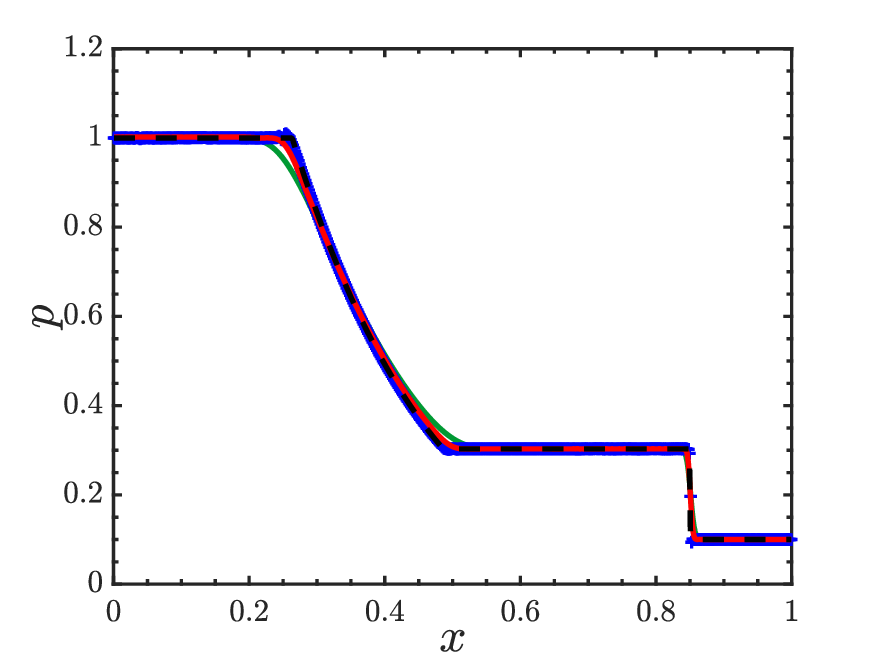}\end{minipage}\quad
\begin{minipage}{0.44\textwidth}\includegraphics[width=\linewidth]{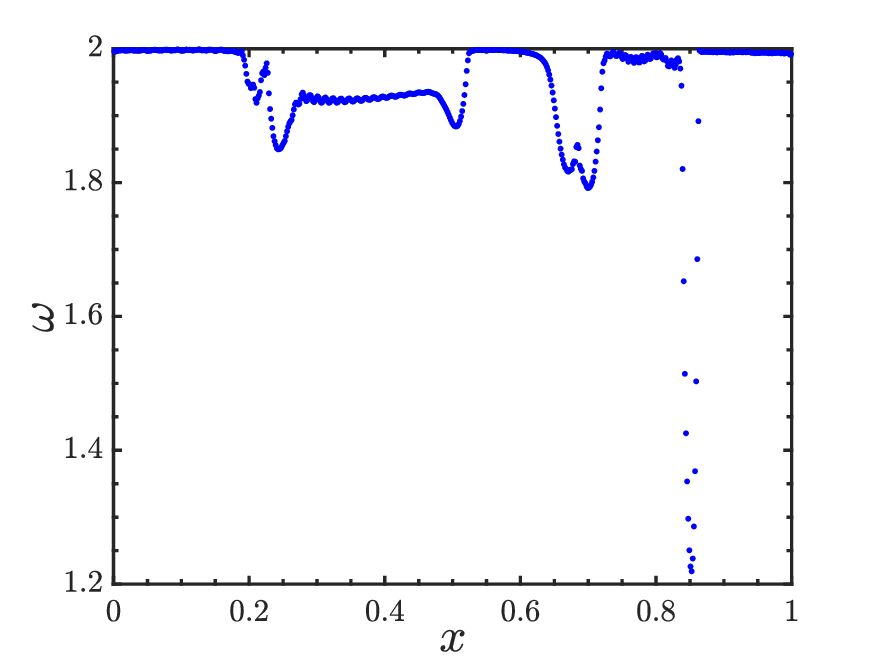}\end{minipage}
\caption{Example 5.6: Sod shock-tube problem at $t=0.2$ with $N_x=600$: density, velocity, pressure, and adaptive relaxation parameter.}\label{fig:sod}
\end{figure}
\begin{figure}[htbp]\centering
\begin{subfigure}{0.32\textwidth}\includegraphics[width=\linewidth]{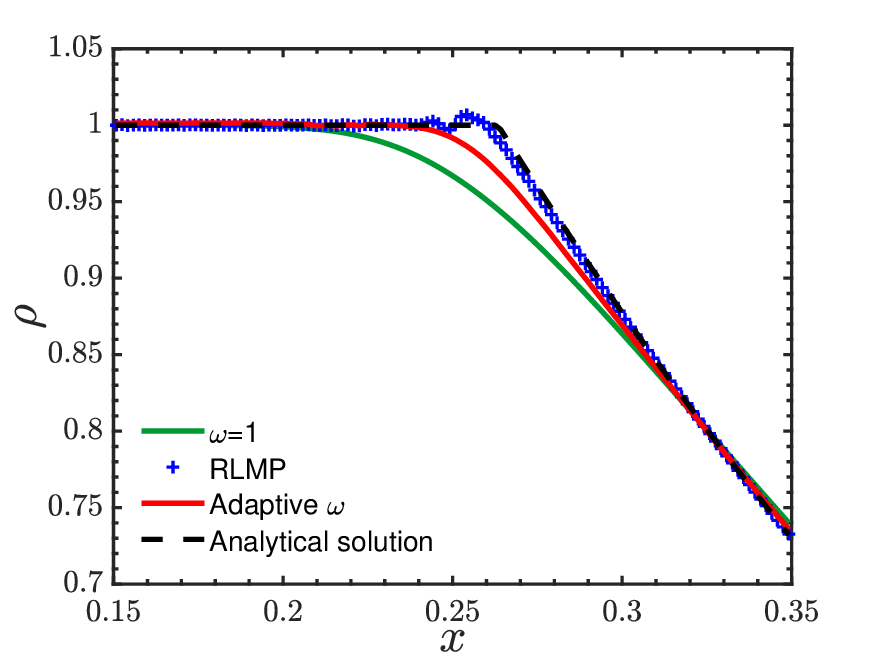}\end{subfigure}\hfill
\begin{subfigure}{0.32\textwidth}\includegraphics[width=\linewidth]{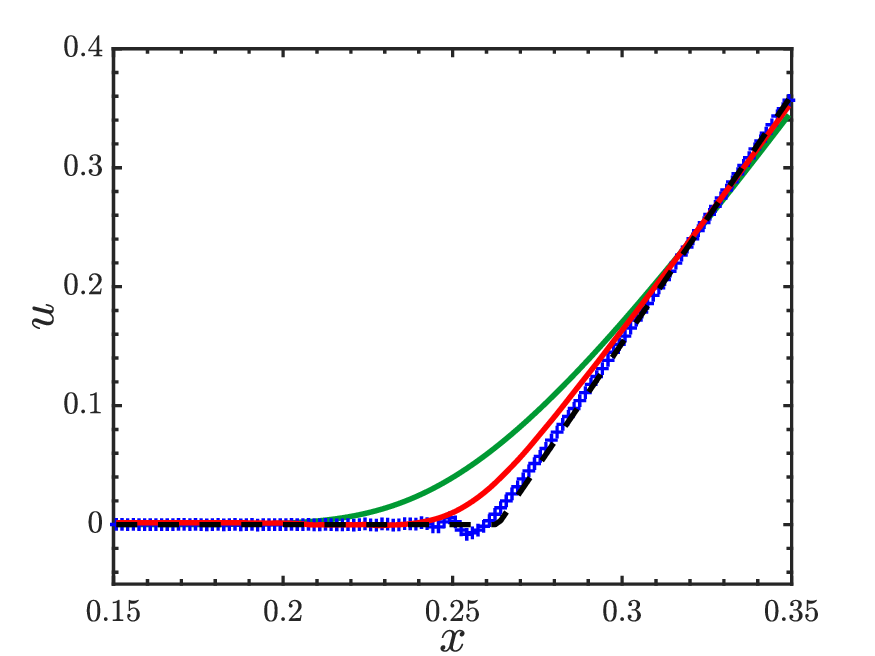}\end{subfigure}\hfill
\begin{subfigure}{0.32\textwidth}\includegraphics[width=\linewidth]{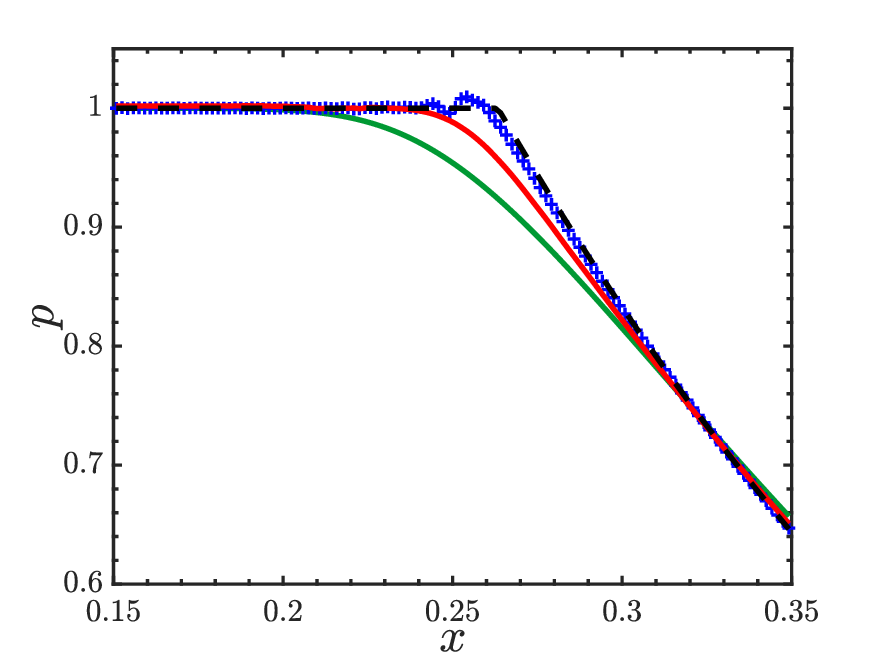}\end{subfigure}
\caption{Example 5.6: Enlarged density, velocity, and pressure profiles near the head of the Sod rarefaction wave.}\label{fig:sod_duibi}
\end{figure}
\begin{table}[htbp]\centering
\begin{tabular}{ccccc}\toprule
Method&$TV(\rho)$&$TV(u)$&$TV(p)$&CPU time (s)\\\midrule
RLMP LBM \cite{WissocqLiuAbgrall2025}&0.3063&0.4031&0.4073&0.1709\\
Present LBM&0.2718&0.3567&0.3570&0.1293\\\bottomrule
\end{tabular}
\caption{Example 5.6: Local total variation and measured CPU time for the reproduced RLMP LBM \cite{WissocqLiuAbgrall2025} and the present LBM over the domain $I=[0.15,0.35]$.}\label{tab:sod_tv_cpu}
\end{table}

\noindent\textbf{Example 5.7 (Shu--Osher problem).}
The Shu--Osher problem \cite{ShuOsher1989} examines the interaction of a shock with a smooth entropy wave. Figure~\ref{fig:shuosher} shows that the fixed-$\omega=1$ scheme strongly damps the post-shock oscillations, whereas the RLMP LBM and the present LBM retain the principal oscillatory structures and agree well with the WENO reference. The enlarged profiles in Figure~\ref{fig:shu_duibi} reveal a small, localized pointwise deviation in the result of the reproduced RLMP LBM near the shock-transition cell, most clearly in velocity and pressure. By comparison, the present LBM remains smooth and close to the reference solution. These results indicate that adaptive relaxation stabilizes the shock locally without imposing excessive dissipation on the smooth oscillatory region.
\begin{figure}[htbp]\centering
\begin{minipage}{0.44\textwidth}\includegraphics[width=\linewidth]{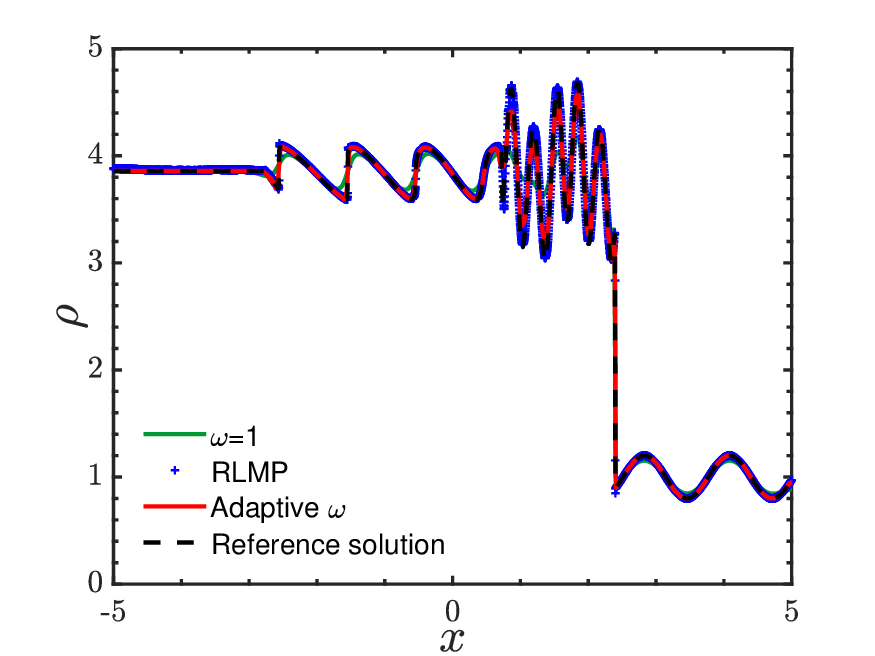}\end{minipage}\quad
\begin{minipage}{0.44\textwidth}\includegraphics[width=\linewidth]{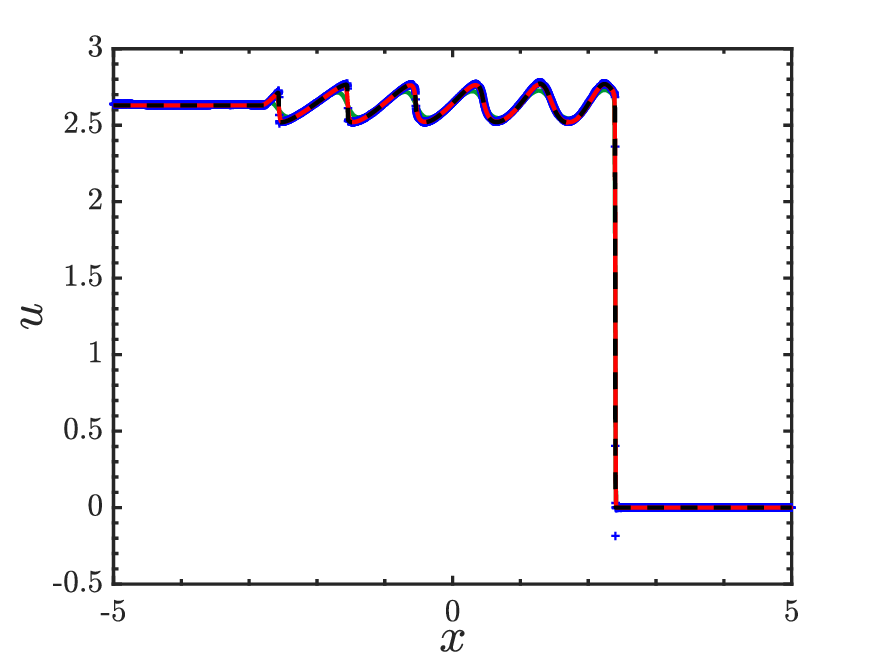}\end{minipage} \\
\begin{minipage}{0.44\textwidth}\includegraphics[width=\linewidth]{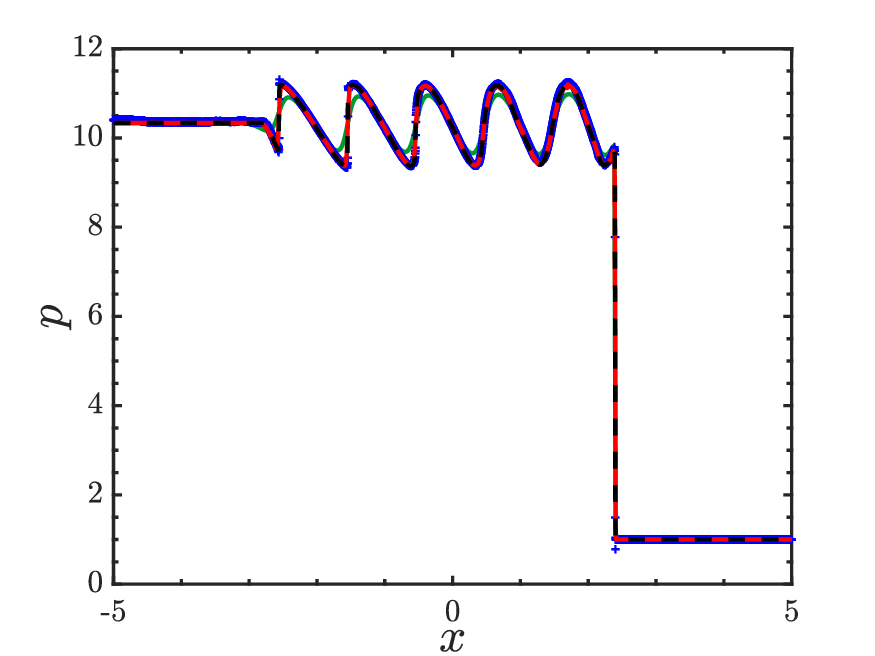}\end{minipage}\quad
\begin{minipage}{0.44\textwidth}\includegraphics[width=\linewidth]{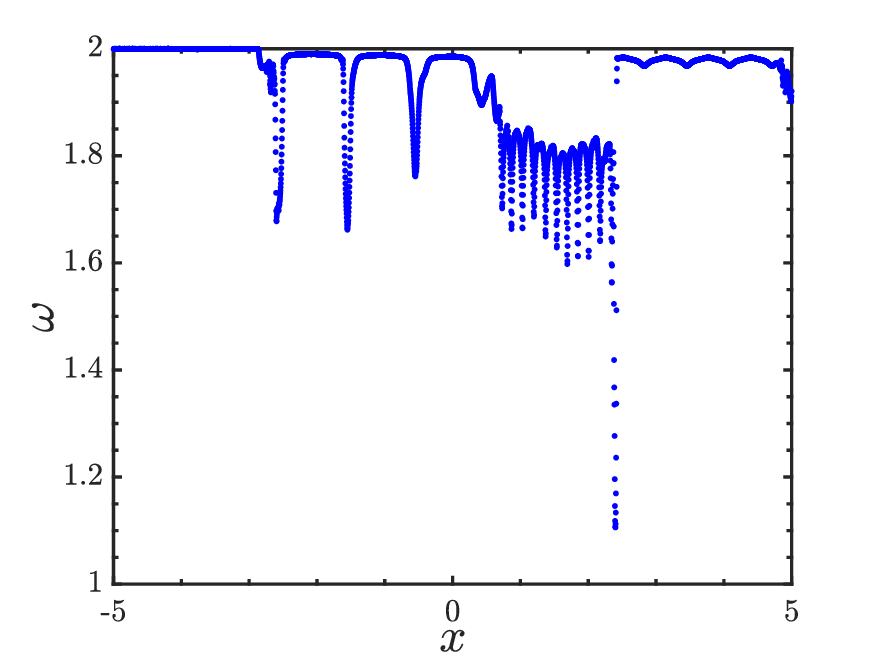}\end{minipage}
\caption{Example 5.7: Shu--Osher problem at $t=1.8$ with $N_x=4000$: density, velocity, pressure, and adaptive relaxation parameter.}\label{fig:shuosher}
\end{figure}
\begin{figure}[htbp]\centering
\begin{subfigure}{0.325\textwidth}\includegraphics[width=\linewidth]{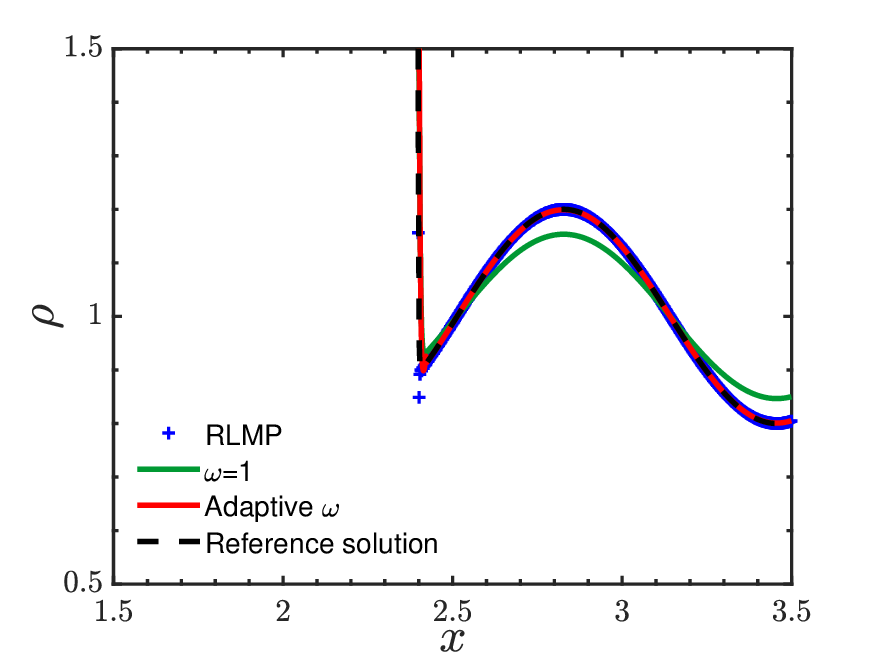}\end{subfigure}
\begin{subfigure}{0.325\textwidth}\includegraphics[width=\linewidth]{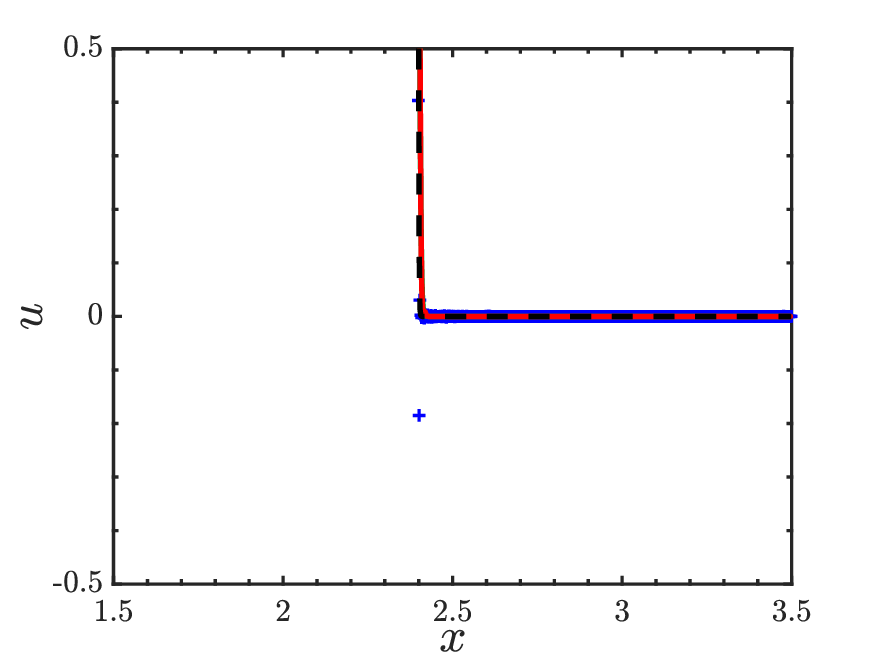}\end{subfigure}
\begin{subfigure}{0.325\textwidth}\includegraphics[width=\linewidth]{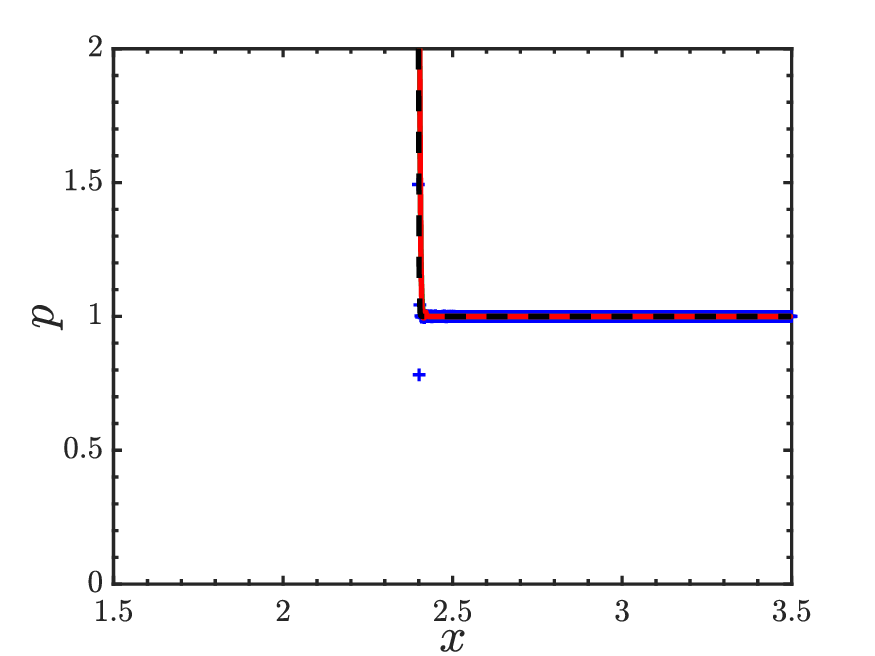}\end{subfigure}
\caption{Example 5.7: Enlarged density, velocity, and pressure profiles in the transmitted-shock region $x\in[1.5,3.5]$.}\label{fig:shu_duibi}
\end{figure}

\noindent\textbf{Example 5.8 (2D Riemann problem).}
We next consider the 2D Riemann problem in \cite{LaxLiu1998} to examine genuinely multidimensional shock interactions. The computational domain is $[0,1]^2$, and the initial states are
\begin{equation}
(\rho,u,v,p)(x,y,0)=\begin{cases}
(1.1,0,0,1.1),&x>0.5,\ y>0.5,\\
(0.5065,0.8939,0,0.35),&x\le0.5,\ y>0.5,\\
(1.1,0.8939,0.8939,1.1),&x\le0.5,\ y\le0.5,\\
(0.5065,0,0.8939,0.35),&x>0.5,\ y\le0.5.
\end{cases}
\end{equation}
We set $N_x=N_y=800$, impose zero-gradient boundary conditions, and use $\mathrm{CFL}=0.9$. Figure~\ref{fig:riemann2d_config4} presents the solution at $t=0.25$. 
It can be seen that the present LBM resolves the straight shocks generated by the initial discontinuities and the curved shocks surrounding the lens-shaped high-pressure region without significant spurious oscillations. Regions of reduced $\omega$ track the shocks and strongest wave interactions, whereas $\omega$ remains close to $2$ in smoother regions. This distribution confirms that the multidimensional stabilization is spatially localized.
\begin{figure}[htbp]\centering
\begin{minipage}{0.45\textwidth}\includegraphics[width=\linewidth]{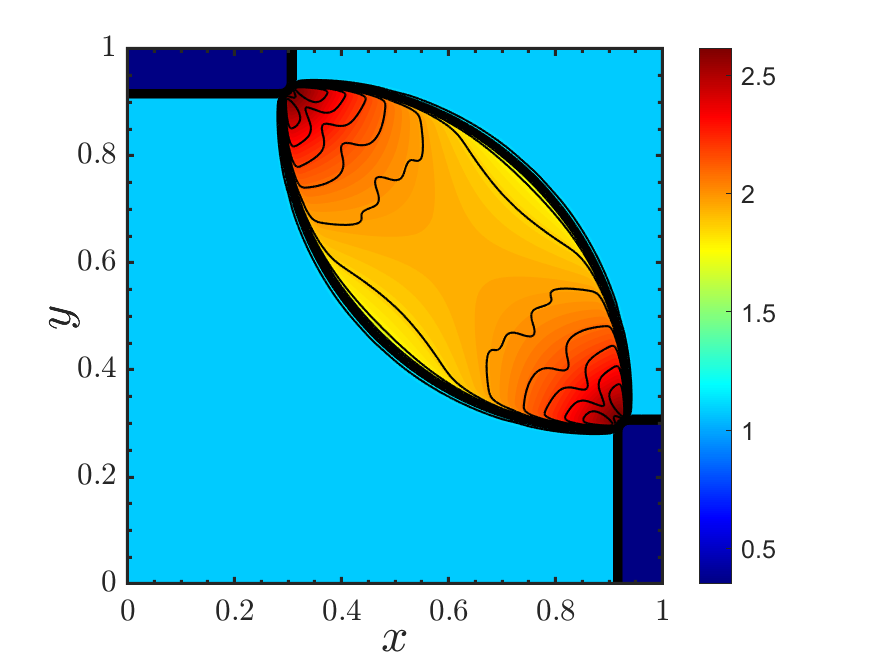}\end{minipage}\quad
\begin{minipage}{0.45\textwidth}\includegraphics[width=\linewidth]{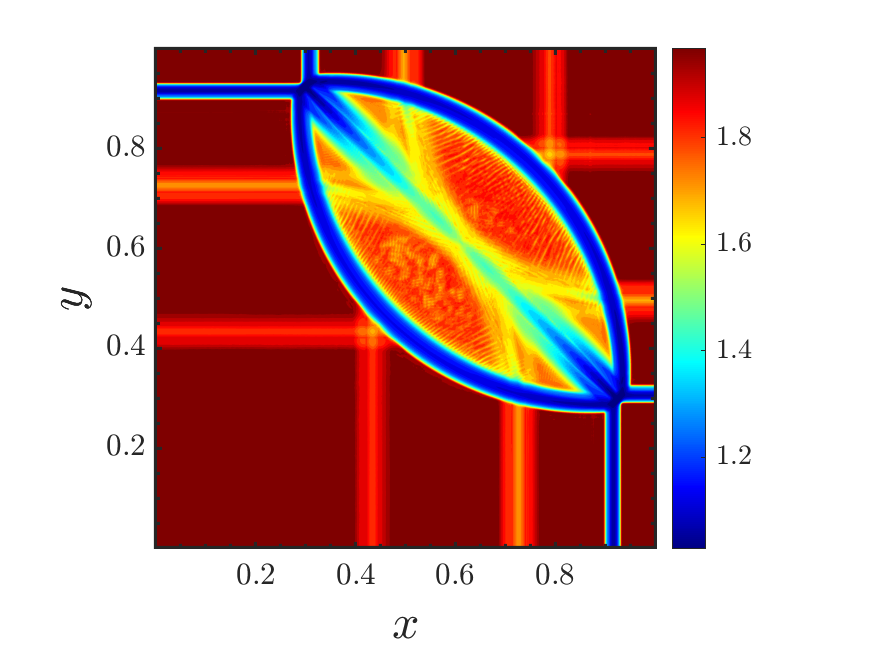}\end{minipage}
\caption{Example 5.8: 2D Riemann problem at $t=0.25$ with $N_x=N_y=800$. Left: pressure field with 29 superimposed density contours at levels $0.52:0.05:1.92$. Right: adaptive relaxation parameter.}
\label{fig:riemann2d_config4}
\end{figure}

\subsection{Shallow-water equations}
We next apply the present LBM to the classical 1D dry-bed and 2D circular dam-break problems for the shallow-water equations \cite{Toro2009}. They examine the treatment of wet--dry fronts, rarefaction waves, and multidimensional wave propagation. In two dimensions, the governing equations are
\[
U_t+F(U)_x+G(U)_y=0,
\qquad
U=(h,hu,hv)^{T},
\]
where
\[
F(U)=
\begin{pmatrix}
	hu\\[1mm]
	hu^2+\dfrac{1}{2}gh^2\\[1mm]
	huv
\end{pmatrix},
\qquad
G(U)=
\begin{pmatrix}
	hv\\[1mm]
	huv\\[1mm]
	hv^2+\dfrac{1}{2}gh^2
\end{pmatrix}.
\]
Here, $h$ is the water depth, and $g=9.81$ is the gravitational acceleration. The 1D equations are obtained by omitting the $y$-flux and the transverse momentum component.  Transmissive boundary conditions are imposed in both examples.

\noindent\textbf{Example 5.9 (1D dam break over a dry bed).}
For the 1D dam break problem over a dry bed, the domain is $x\in[-1,1]$ and the initial data are
\[
(h,u)(x,0)=\begin{cases}(2,0),&x\le0,\\(0,0),&x>0.\end{cases}
\]
The computation is performed to $t=0.15$ with $N_x=600$ and $\mathrm{CFL}=0.9$. Figure~\ref{fig:dry_bed} compares the numerical water depth and velocity with the analytical dry-bed dam-break solution. Both variables agree well with the analytical profiles in the rarefaction region, and the moving wet--dry front is captured without visible nonphysical oscillations. The reduction of $\omega$ is confined mainly to the front and the steep portion of the rarefaction, indicating that the present LBM supplies additional dissipation only where needed.
\begin{figure}[htbp]\centering
\begin{subfigure}{0.325\textwidth}\includegraphics[width=\linewidth]{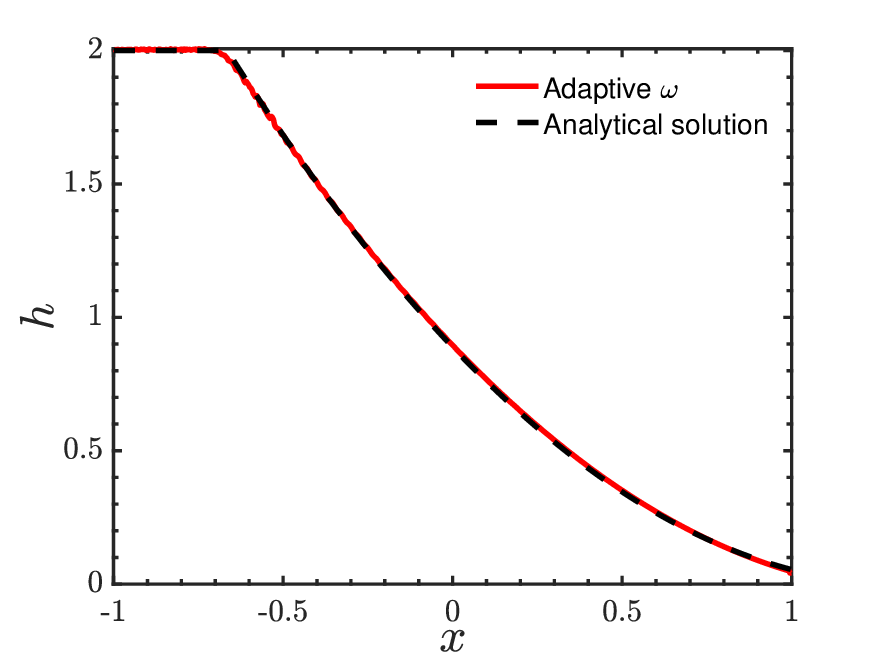}\end{subfigure}
\begin{subfigure}{0.325\textwidth}\includegraphics[width=\linewidth]{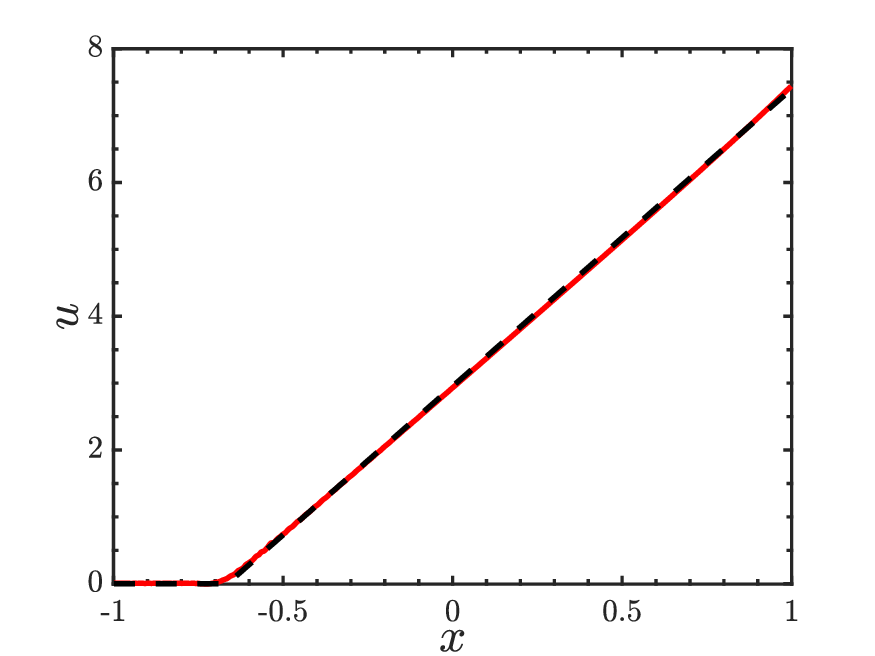}\end{subfigure}
\begin{subfigure}{0.325\textwidth}\includegraphics[width=\linewidth]{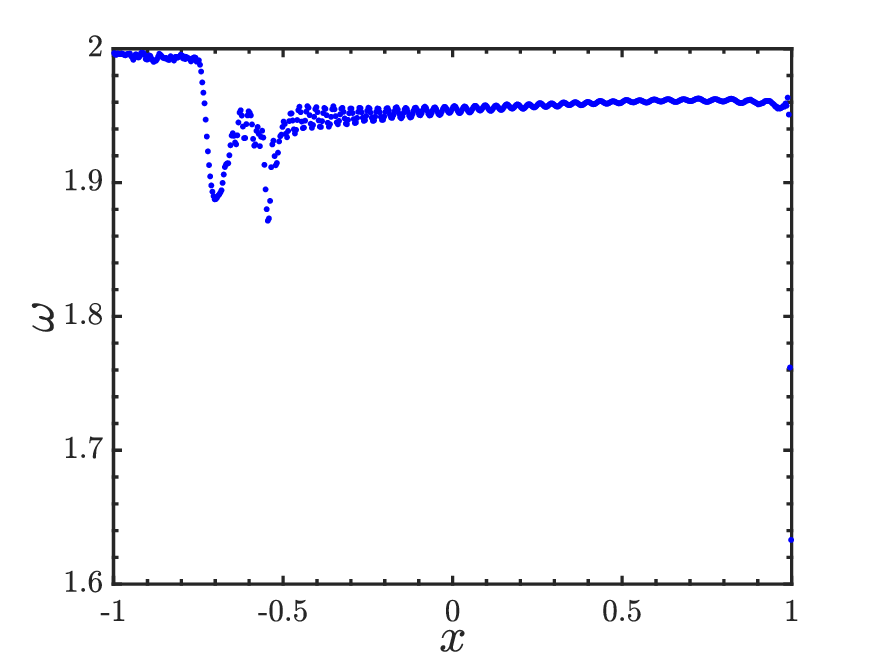}\end{subfigure}
\caption{Example 5.9: 1D dam break over a dry bed at $t=0.15$: water depth $h$, velocity $u$, and adaptive relaxation parameter $\omega$.}\label{fig:dry_bed}
\end{figure}

\noindent\textbf{Example 5.10 (2D circular dam break).}
In this 2D circular dam break problem, the domain is $[-1,1]^2$ and the initial data are
\[
h(x,y,0)=\begin{cases}2,&r\le0.5,\\1,&r>0.5,\end{cases}\qquad r=\sqrt{x^2+y^2},\qquad u(x,y,0)=v(x,y,0)=0.
\]
The computation is performed to $t=0.07$ with $N_x=N_y=400$ and $\mathrm{CFL}=0.4$. A high-resolution radial solution is used as the reference for the centerline comparison along $y=0$. Figure~\ref{fig:circular_dam} shows that the computed water depth $h$ agrees well with the reference solution, resolving both the outward-propagating front and the inner rarefaction structure. The velocity also follows the reference solution through the principal acceleration and deceleration regions. Despite the use of a Cartesian grid, the 2D depth contour remains nearly circular. The region of reduced $\omega$ follows the propagating front and other strong gradients, while $\omega$ remains close to $2$ elsewhere. The present LBM therefore combines good symmetry preservation with spatially localized dissipation in this multidimensional dam-break problem.
\begin{figure}[htbp]\centering
\begin{subfigure}{0.44\textwidth}\includegraphics[width=\linewidth]{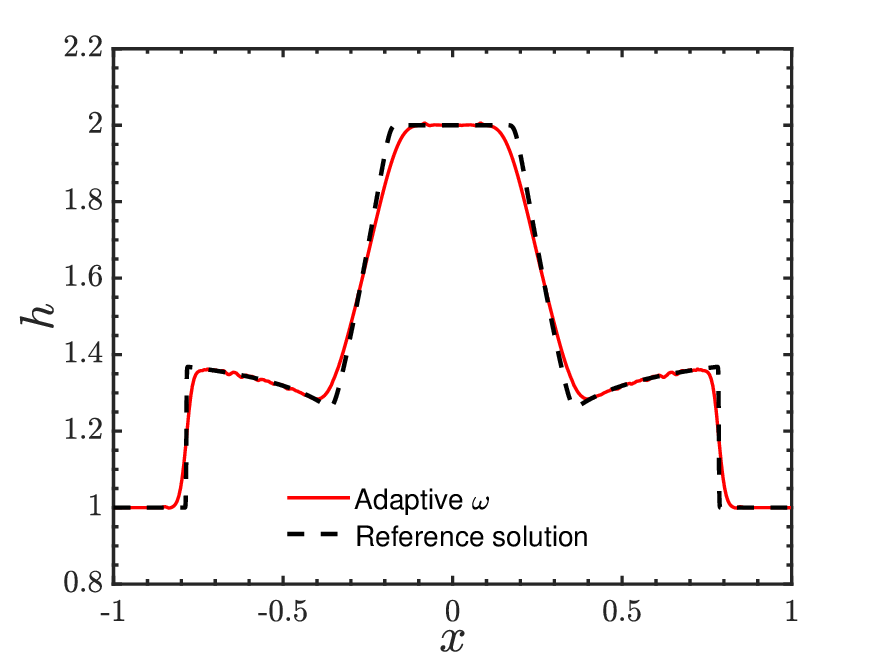}\end{subfigure}\quad
\begin{subfigure}{0.44\textwidth}\includegraphics[width=\linewidth]{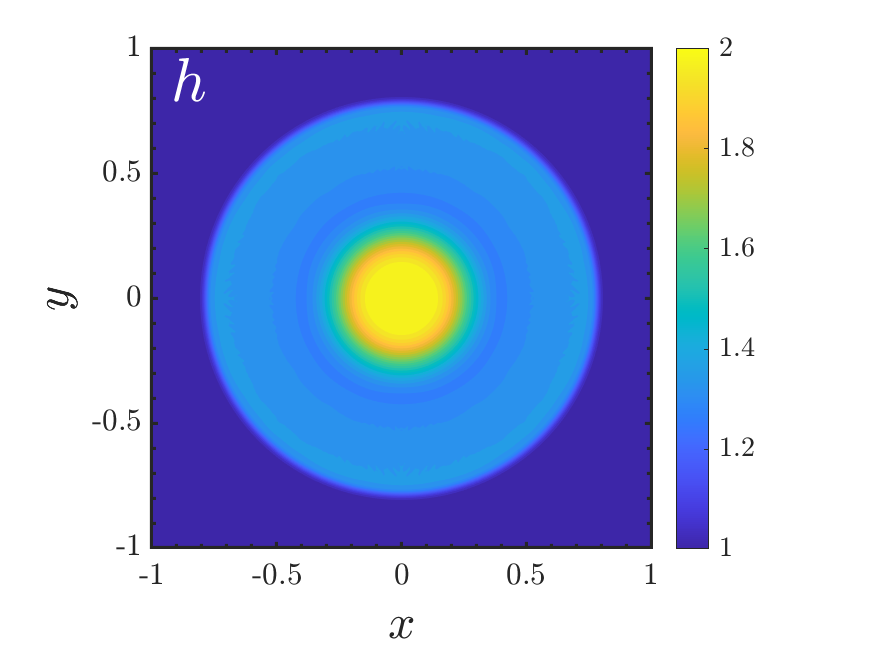}\end{subfigure} \\
\begin{subfigure}{0.44\textwidth}\includegraphics[width=\linewidth]{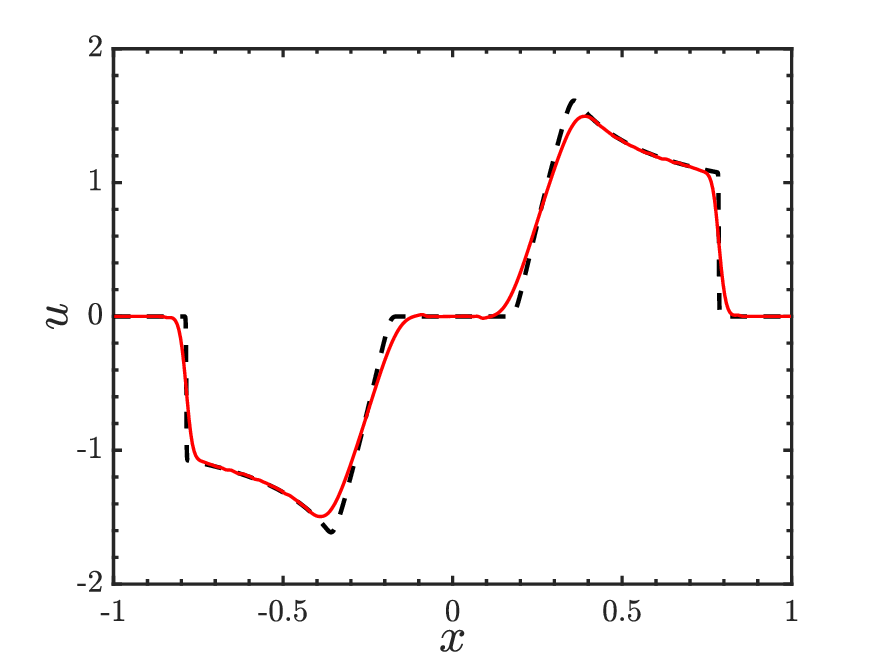}\end{subfigure}\quad
\begin{subfigure}{0.44\textwidth}\includegraphics[width=\linewidth]{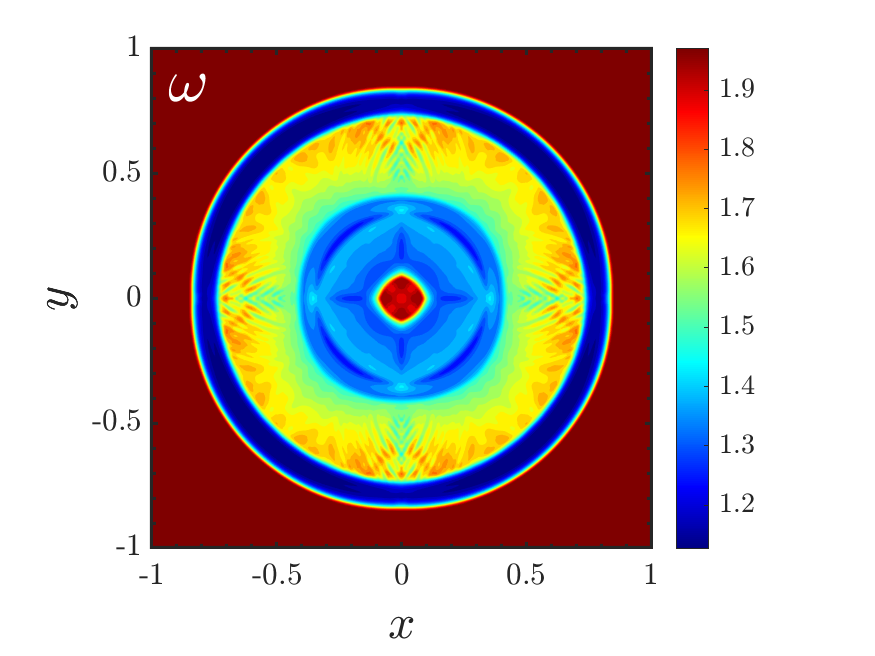}\end{subfigure}
\caption{Example 5.10: 2D circular dam break at $t=0.07$ with $N_x=N_y=400$: centerline depth and velocity compared with a high-resolution radial reference solution, depth contour, and adaptive relaxation parameter.}\label{fig:circular_dam}
\end{figure}

\subsection{Reactive Euler equations}
Finally, we apply the present LBM to the reactive Euler equations to examine its robustness for detonation waves and reactive discontinuities. In two dimensions, the governing equations are \cite{Chorin1977,OranBoris2001}
\[
U_t+F(U)_x+G(U)_y=S(U),\qquad U=(\rho,\rho u,\rho v,\rho E,\rho Y)^T,
\]
with
\[
F(U)=\begin{pmatrix}\rho u\\\rho u^2+p\\\rho uv\\u(\rho E+p)\\\rho uY\end{pmatrix},\qquad
G(U)=\begin{pmatrix}\rho v\\\rho uv\\\rho v^2+p\\v(\rho E+p)\\\rho vY\end{pmatrix},
\]
\[
\rho E=\frac{p}{\gamma-1}+\frac12\rho(u^2+v^2)+\rho qY,
\qquad S(U)=(0,0,0,0,-R)^T,
\]
where $R=K\rho Y\exp(-T_i/\mathcal{T})$ is the reaction rate in the Arrhenius form and $\mathcal{T}=p/\rho$ is the temperature. The 1D equations are obtained by omitting the $y$-flux and the transverse momentum component.

\noindent\textbf{Example 5.11 (1D stable detonation).}
We first simulate the stable detonation problem on $x\in[0,100]$, with $K=145.68913$, $T_i=50$, and $q=50$  \cite{LianXu2000}. The initial data are
\[
(\rho,u,p,Y)(x,0)=\begin{cases}(2.5706,6.86089,80.02055,0),&x<2,\\(1,0,1,1),&x\ge2,\end{cases}
\]
and the terminal time is $t=7.5$.
We use $\delta_x=1/150$, $\mathrm{CFL}=0.9$, and impose transmissive boundary conditions. The reference solution is computed on the same mesh using a fifth-order WENO scheme with global Lax--Friedrichs flux splitting and third-order TVD Runge--Kutta time integration. Figure~\ref{fig:Stable_detonation_1D} shows close agreement between the present LBM and the WENO reference for density, velocity, pressure, and reactant mass fraction. In particular, the detonation front and reaction zone remain sharply resolved. The parameter $\omega$ decreases locally near the front and other steep gradients but returns close to $2$ in smooth regions. 
\begin{figure}[htbp]\centering
\begin{subfigure}{0.325\textwidth}\includegraphics[width=\linewidth]{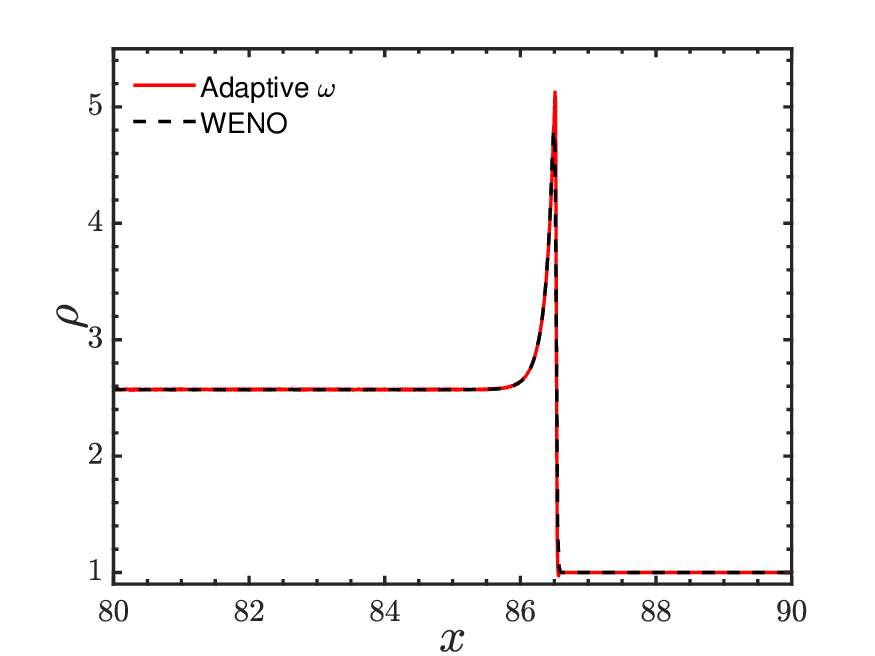}\end{subfigure}
\begin{subfigure}{0.325\textwidth}\includegraphics[width=\linewidth]{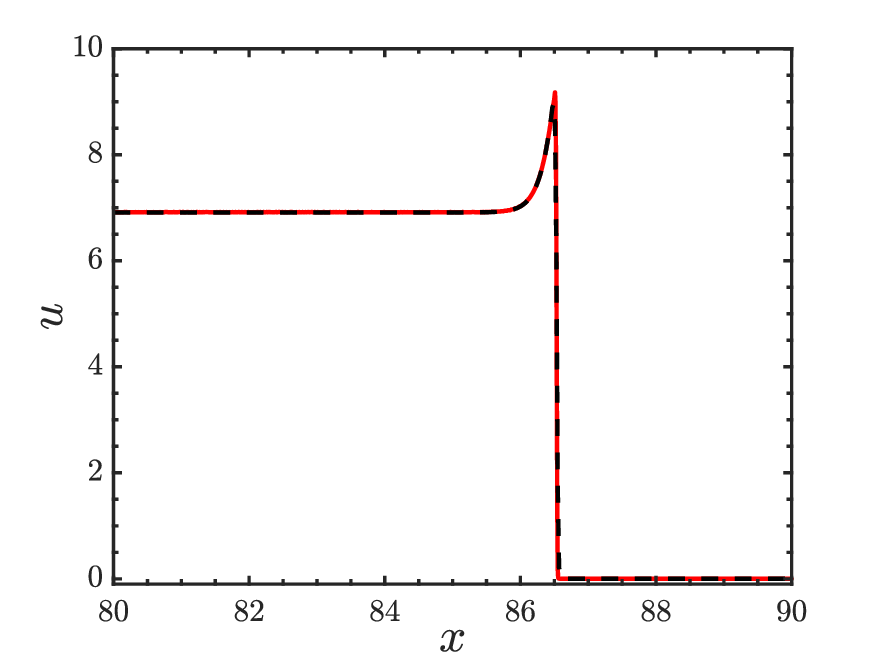}\end{subfigure}
\begin{subfigure}{0.325\textwidth}\includegraphics[width=\linewidth]{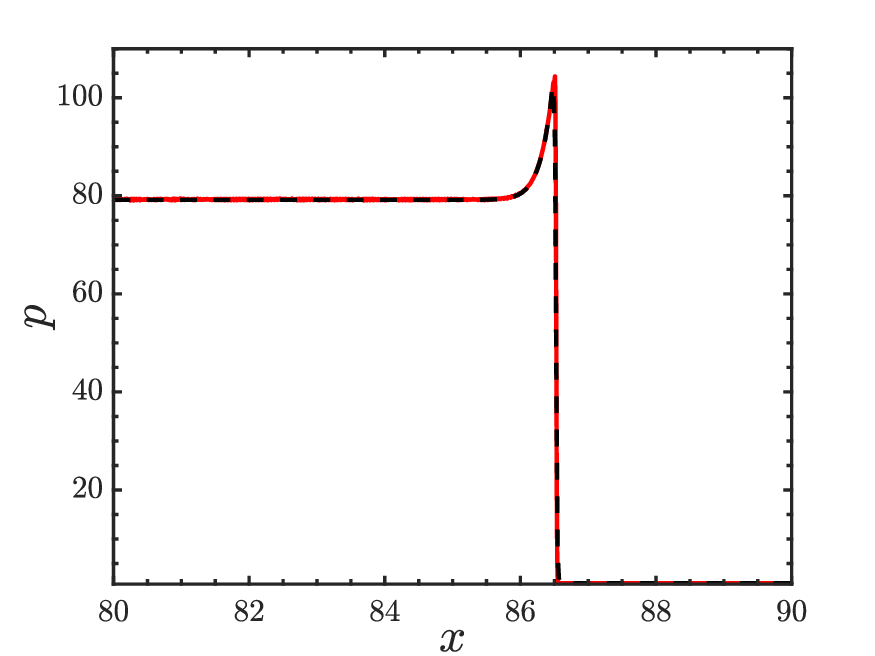}\end{subfigure}\\
\begin{subfigure}{0.325\textwidth}\includegraphics[width=\linewidth]{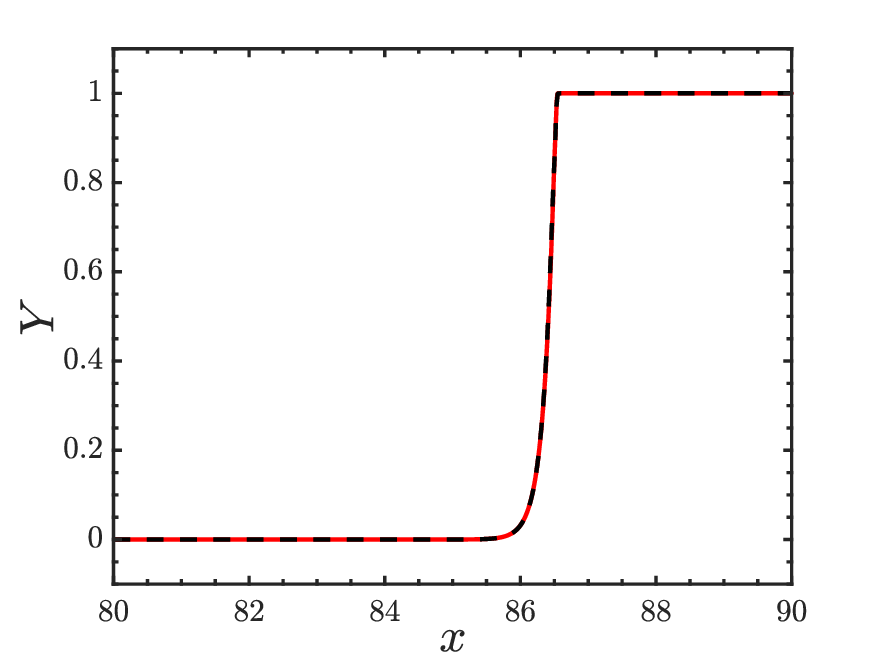}\end{subfigure}
\begin{subfigure}{0.325\textwidth}\includegraphics[width=\linewidth]{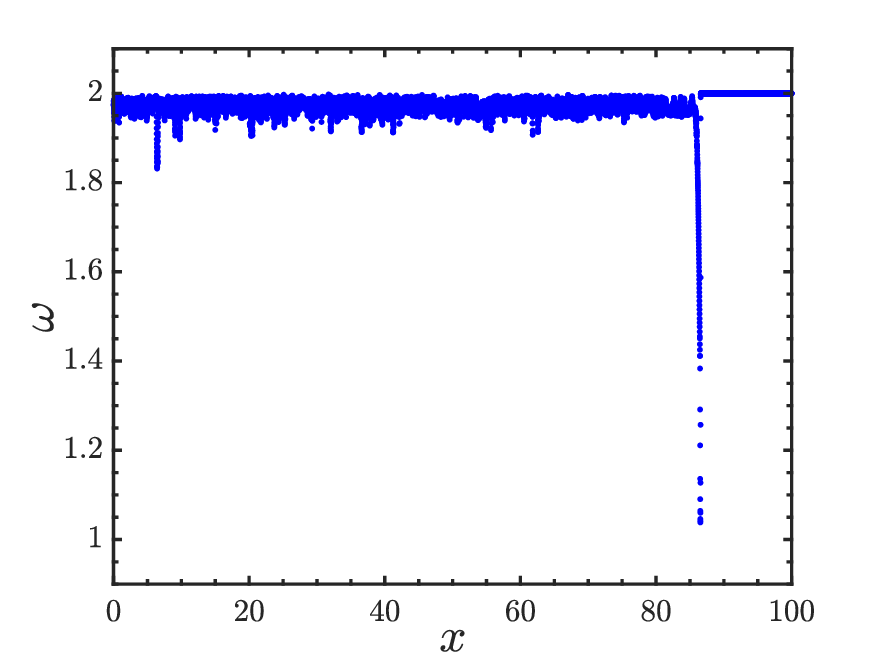}\end{subfigure}
\caption{Example 5.11: Density, velocity, pressure, reactant mass fraction, and adaptive relaxation parameter for the present LBM and WENO solutions of the 1D stable detonation wave at $t=7.5$ with $\delta_x=1/150$.}\label{fig:Stable_detonation_1D}
\end{figure}

\noindent\textbf{Example 5.12 (2D reactive circular discontinuity).}
In this example, the spatial domain is $[-1,1]^2$. We take $\gamma=1.2$, $K=2566.4$, $T_i=50$, and $q=50$, with the following initial data \cite{ZhaoHuang2020,Zhao2022MC}:
\[
(\rho,u,v,p,Y)(x,y,0)=\begin{cases}(1,0,0,80,0.2),&x^2+y^2\le0.36,\\(1,0,0,10,0.8),&\text{otherwise}.
\end{cases}
\]
The computation is carried out to $t=0.04$ on a uniform grid with $N_x=N_y=400$. We take $\mathrm{CFL}=0.9$ and use transmissive boundary conditions on all sides. Figure~\ref{fig:Circular_Reactive_Explosion2D} shows that the expanding reactive front remains nearly circular during the early-time evolution. The pressure field depicts the outward expansion of the initially hot region, while the reactant mass fraction distinguishes the burned and unburned regions without visible grid-scale oscillations. The reduction of $\omega$ is concentrated along the expanding front and other strong gradients. These results indicate that the present LBM resolves multidimensional reactive discontinuities while preserving the localized character of its dissipation mechanism.
\begin{figure}[htbp]\centering
\begin{subfigure}{0.325\textwidth}\includegraphics[width=\linewidth]{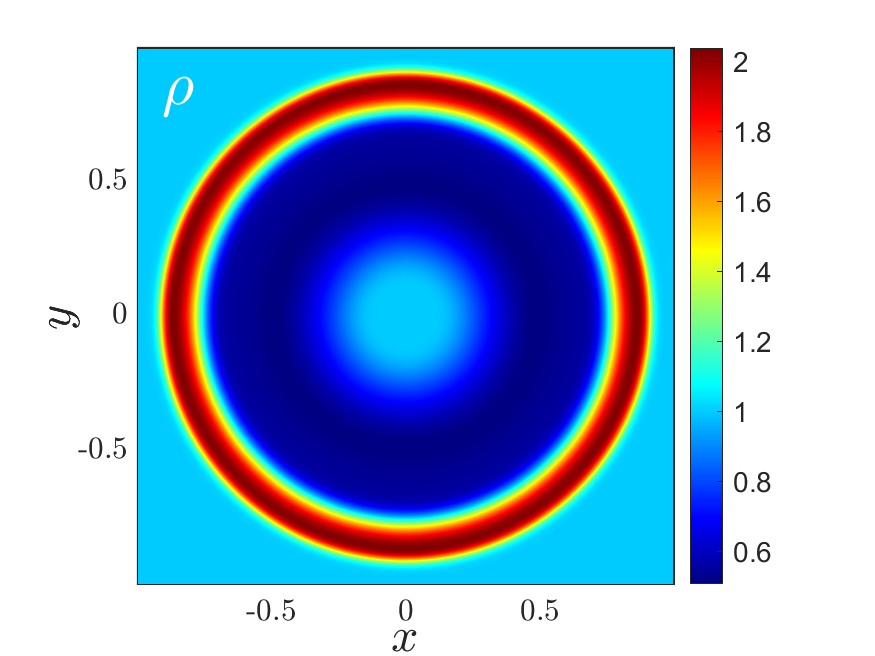}\end{subfigure}
\begin{subfigure}{0.325\textwidth}\includegraphics[width=\linewidth]{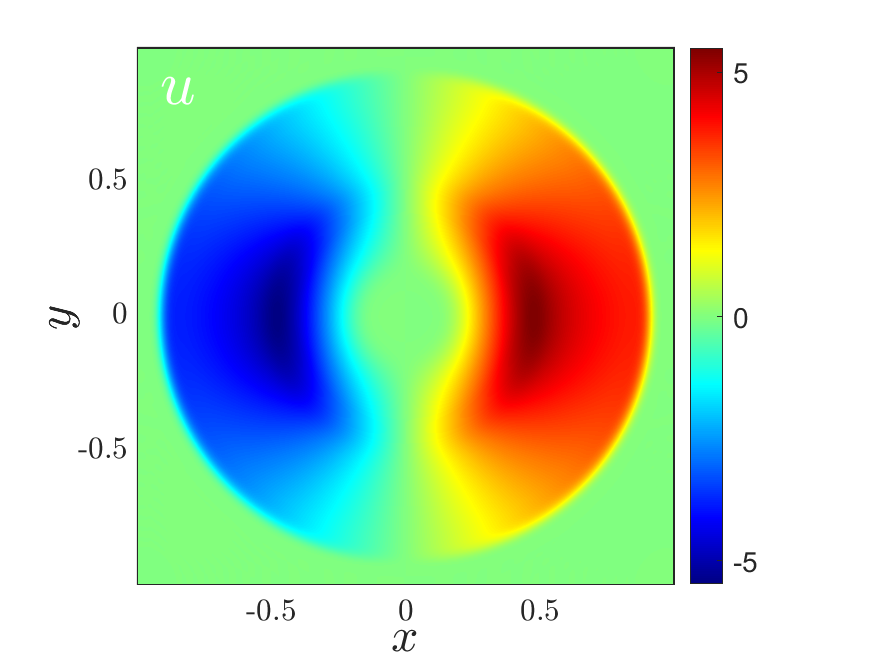}\end{subfigure}
\begin{subfigure}{0.325\textwidth}\includegraphics[width=\linewidth]{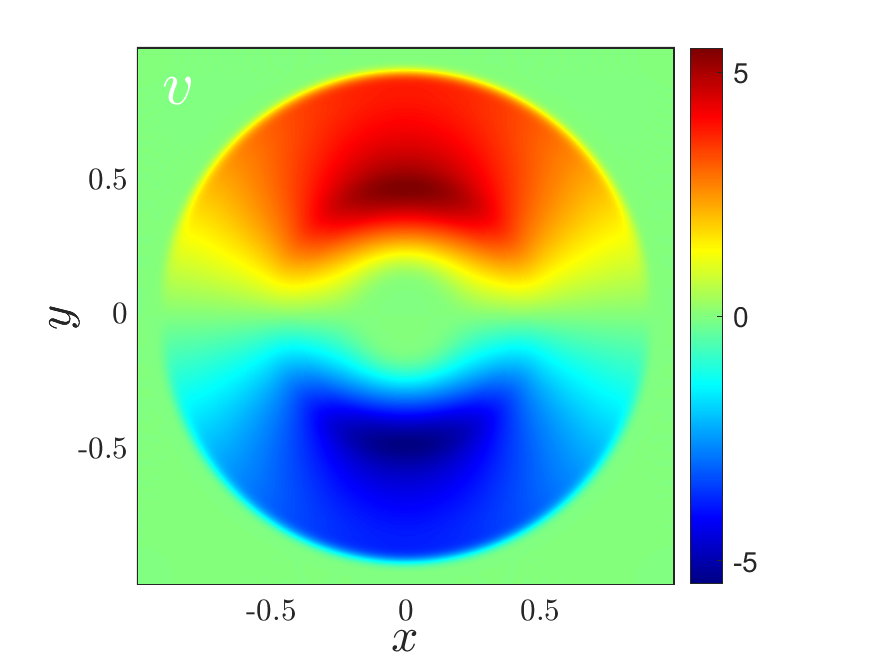}\end{subfigure}\\
\begin{subfigure}{0.325\textwidth}\includegraphics[width=\linewidth]{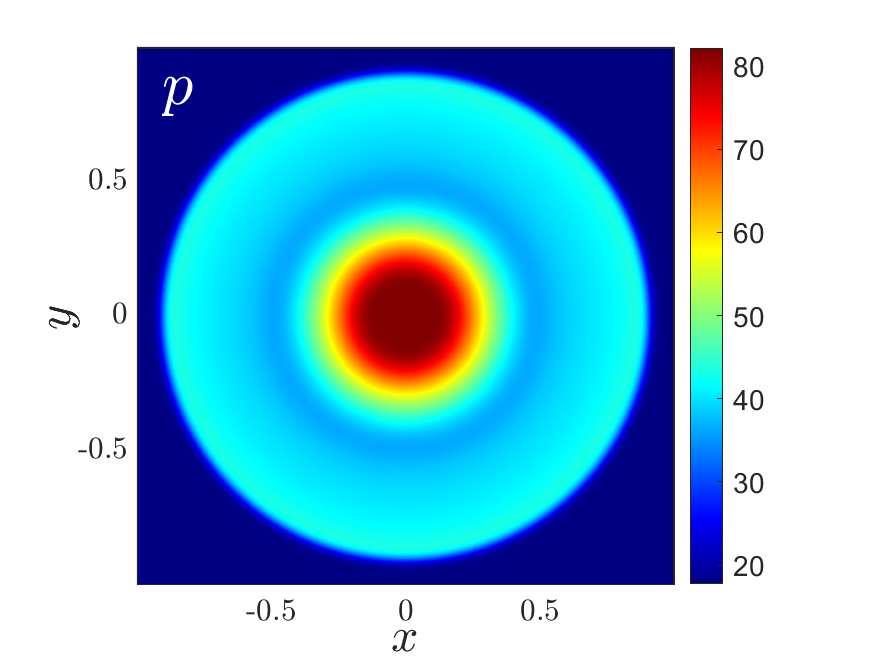}\end{subfigure}
\begin{subfigure}{0.325\textwidth}\includegraphics[width=\linewidth]{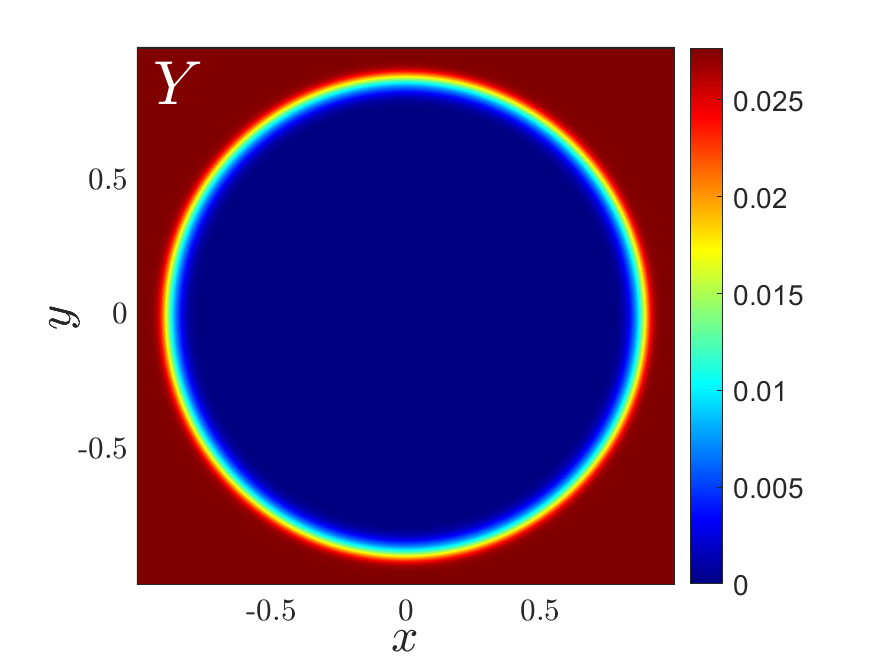}\end{subfigure}
\begin{subfigure}{0.325\textwidth}\includegraphics[width=\linewidth]{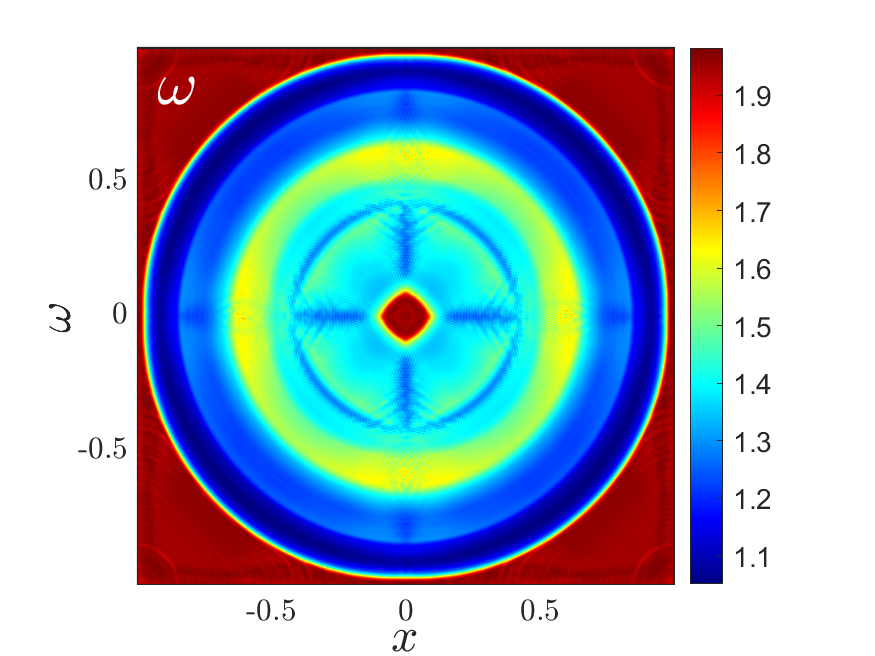}\end{subfigure}
\caption{Example 5.12: Density, velocity components, pressure, reactant mass fraction, and adaptive relaxation parameter for the 2D reactive circular discontinuity at $t=0.04$.}\label{fig:Circular_Reactive_Explosion2D}
\end{figure}

\noindent\textbf{Example 5.13 (2D stable cellular detonation).} 
Stable cellular detonations governed by one-step kinetics have been used extensively to demonstrate multidimensional reactive-flow solvers \cite{PapalexandrisLeonardDimotakis2002,GaoDonLi2012,LiGaoDon2018}. Here we consider the  cellular detonation problem in \cite{LiGaoDon2018} on the domain $[0,40]\times[-15,15]$. The parameters are $\gamma=1.2$, $q=2$, $T_i=20$, and $K=1.13436364\times10^6$; for these parameters, the underlying 1D detonation is linearly stable \cite{GaoDonLi2012,LiGaoDon2018}. The problem is formulated in a shock-attached frame moving at the prescribed detonation speed $D$. The initial front is located at $x_d=30$. Ahead of the front ($x>x_d$), the uniform unreacted state is
$(\rho,u,v,p,Y)=(1,-D,0,1,1)$; behind it ($x<x_d$), the initial state is the corresponding steady 1D Zeldovich--von Neumann--D\"oring (ZND) profile for the same kinetic parameters and detonation speed. Here $D=\sqrt{f}\,D_{\mathrm{CJ}}$, where $f=1.1$ is the overdrive factor and $D_{\mathrm{CJ}}$ is the Chapman--Jouguet detonation speed,
$$
D_{\mathrm{CJ}}=
\left[
\gamma+(\gamma^2-1)q+
\sqrt{(\gamma^2-1)q
\left(2\gamma+(\gamma^2-1)q\right)}
\right]^{1/2}.
$$
To initiate the transverse dynamics, the velocity perturbation
\[
v(x,y,0)=0.1\sin\left(2\pi k y/30\right),\qquad k=6,
\]
is imposed only in the strip $x_d-1\leq x\leq x_d$; elsewhere, the transverse velocity is initially zero \cite{LiGaoDon2018}.

Reflective boundary conditions are imposed at $y=\pm15$. At the right boundary, the incoming state is fixed at the uniform unreacted state, whereas outgoing disturbances are treated transmissively. The left boundary uses a zero-gradient outflow condition, supplemented by a smooth sponge layer over $0\leq x\leq6$ to reduce spurious reflections. The computation uses a $1200\times900$ uniform grid with $\delta_x=\delta_y=1/30$ and $\mathrm{CFL}=0.3$.

Figure~\ref{fig:stable cellular detonation 1} shows a fully developed cellular detonation at $t=268$. The corrugated leading front and intersecting transverse waves are clearly resolved in the flow fields. The adaptive relaxation parameter decreases near the leading shock, reaction front, transverse waves, and their interaction regions, while remaining close to $2$ in smooth regions. Figure~\ref{fig:stable cellular detonation 2} further shows that the localized pressure maxima and the corrugated reaction front persist and evolve regularly at late times. The repeated diamond-shaped trajectories in Figure~\ref{fig:stable cellular detonation 3} reflect the propagation, reflection, and interaction of transverse waves, indicating a long-time quasi-periodic cellular regime. These structures are qualitatively consistent with results of the high-order hybrid scheme in \cite{LiGaoDon2018}. These results demonstrate that the proposed method can robustly capture long-time multidimensional shock--reaction interactions.
\begin{figure}[htbp]\centering
	\begin{subfigure}{0.325\textwidth}\includegraphics[width=\linewidth]{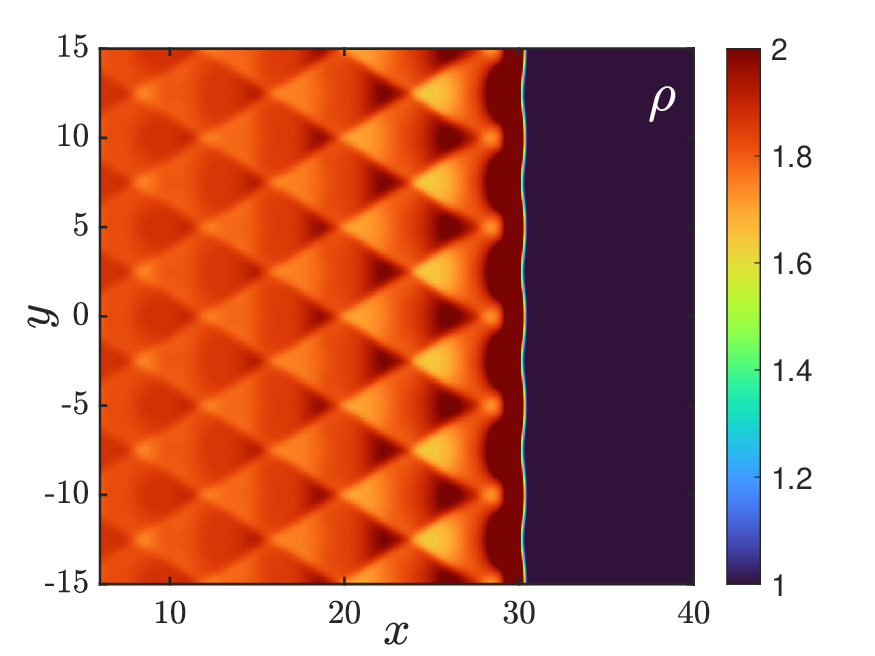}\end{subfigure}
	\begin{subfigure}{0.325\textwidth}\includegraphics[width=\linewidth]{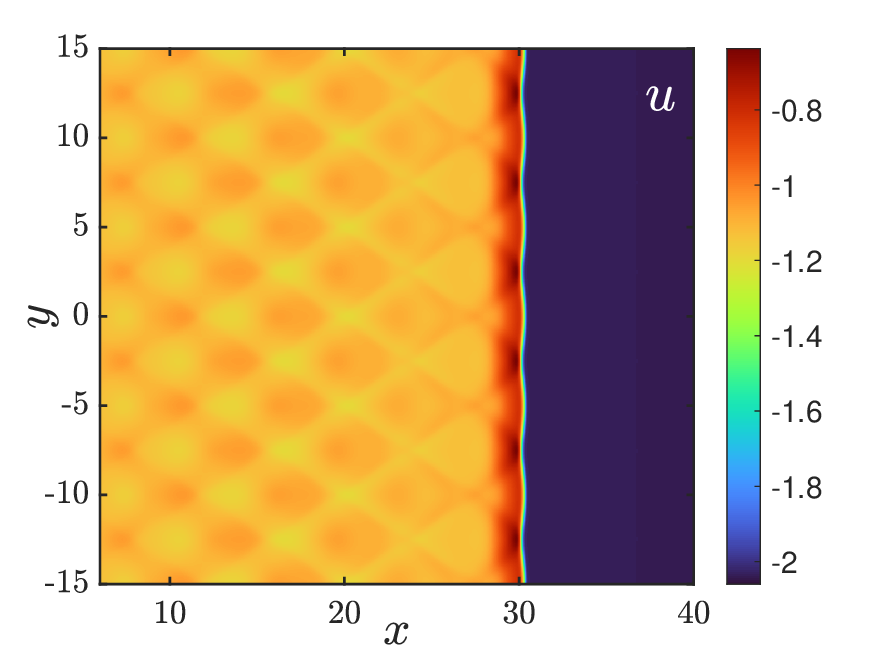}\end{subfigure}
	\begin{subfigure}{0.325\textwidth}\includegraphics[width=\linewidth]{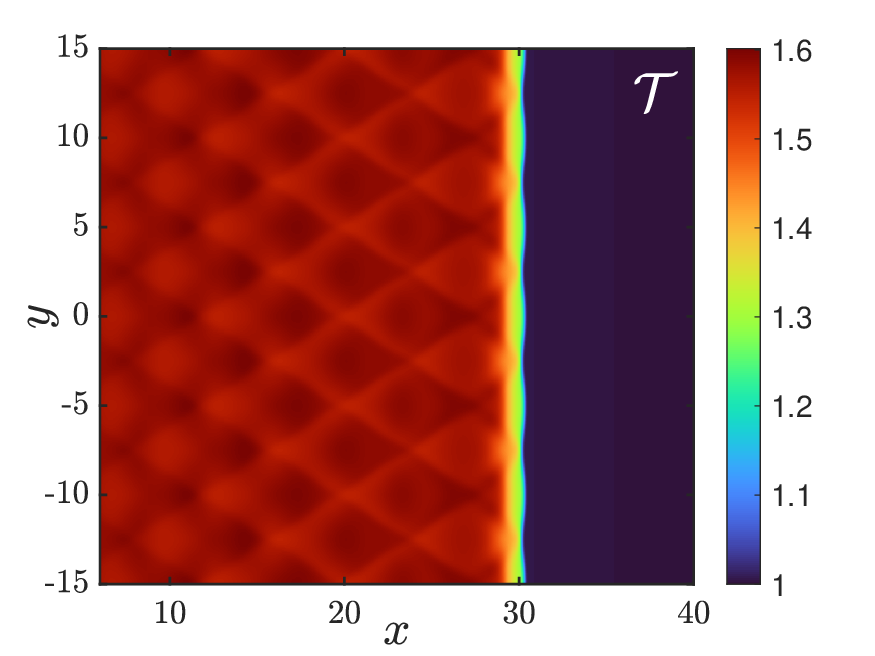}\end{subfigure}\\
	\begin{subfigure}{0.325\textwidth}\includegraphics[width=\linewidth]{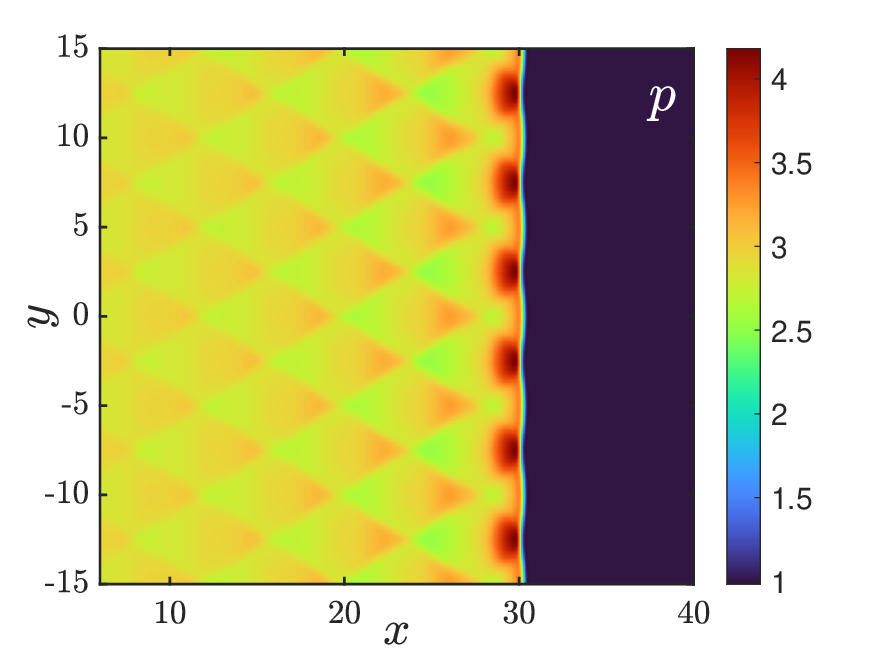}\end{subfigure}
	\begin{subfigure}{0.325\textwidth}\includegraphics[width=\linewidth]{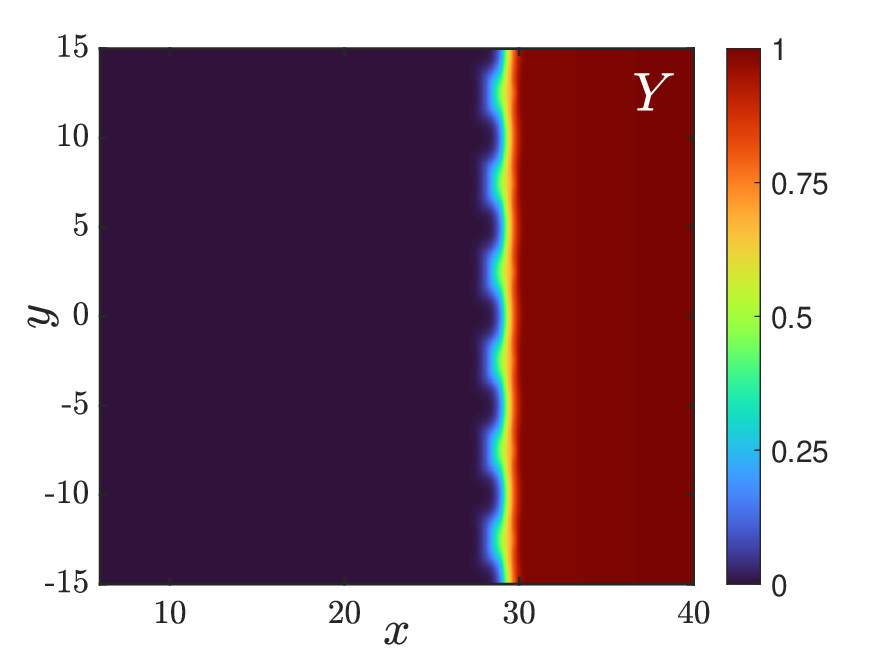}\end{subfigure}
	\begin{subfigure}{0.325\textwidth}\includegraphics[width=\linewidth]{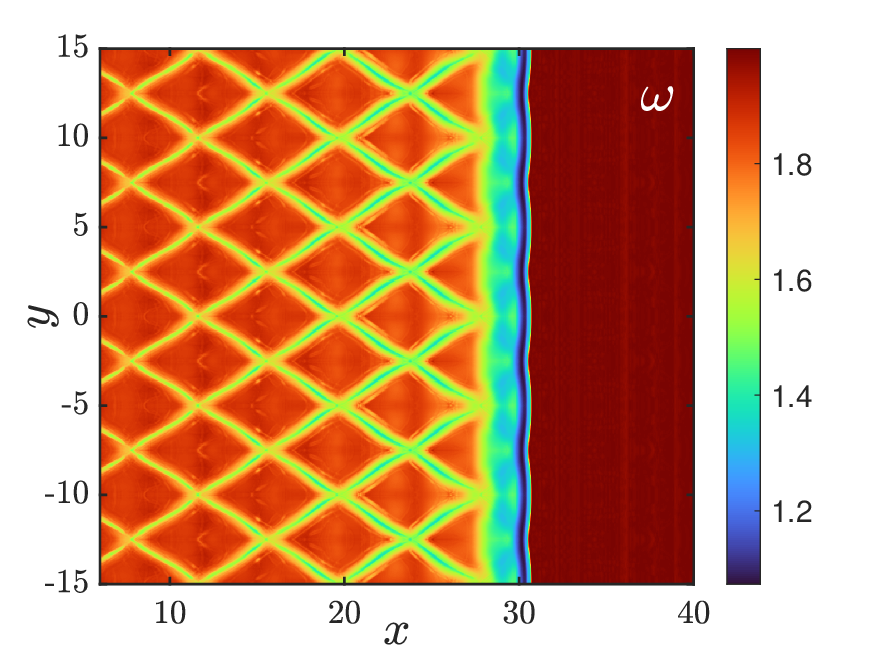}\end{subfigure}
\caption{Example 5.13: Density, streamwise velocity, temperature, pressure, reactant mass fraction, and adaptive relaxation parameter at $t=268$.}\label{fig:stable cellular detonation 1}
\end{figure}

\begin{figure}[htbp]\centering
	\begin{subfigure}{1.0\textwidth}\includegraphics[width=\linewidth]{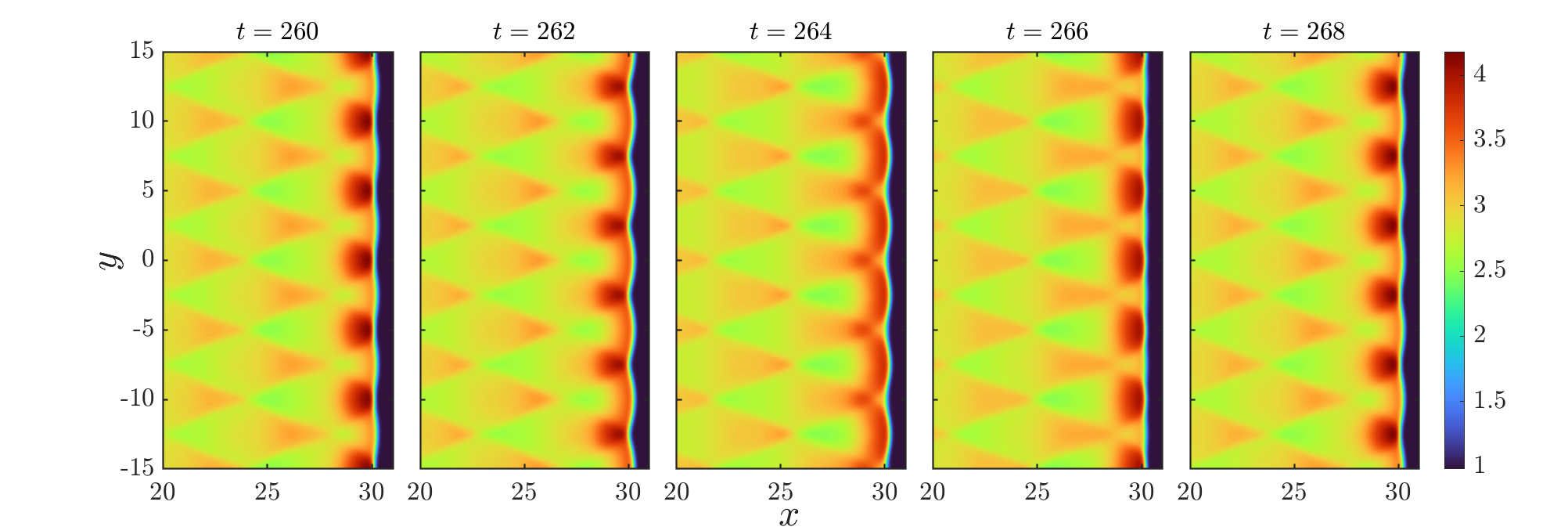}\end{subfigure}\\
	\begin{subfigure}{1.0\textwidth}\includegraphics[width=\linewidth]{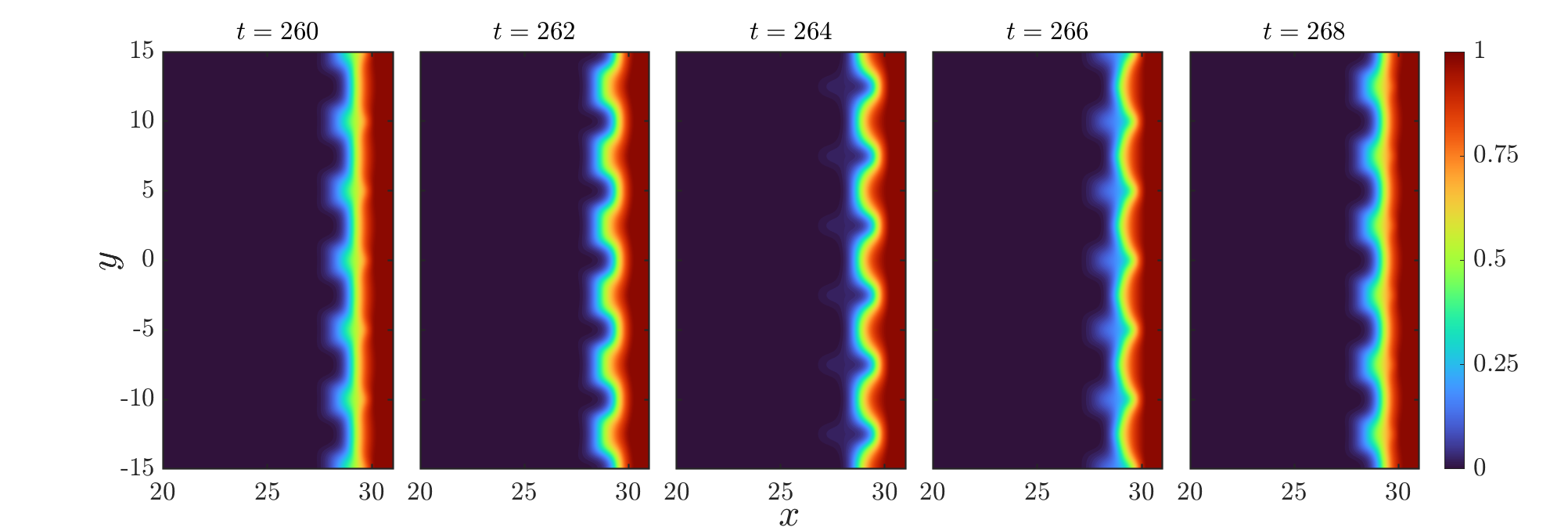}\end{subfigure}
\caption{Example 5.13: Evolution of pressure (top) and reactant mass fraction (bottom)  in the near-front region $20<x<31$ at $t = 260, 262, 264, 266$, and $268$.}\label{fig:stable cellular detonation 2}
\end{figure}

\begin{figure}[htbp]\centering
	\includegraphics[width=1.0\textwidth]{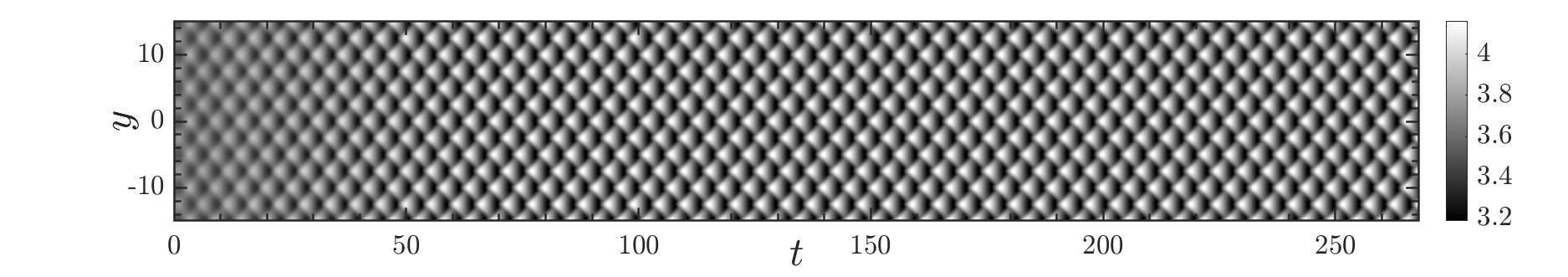}
\caption{Example 5.13: Space-time evolution of the front-region peak pressure $p_{peak}(y,t)=max_{x \in [20,31]} p(x,y,t)$}\label{fig:stable cellular detonation 3}
\end{figure}

\section{Conclusions and remarks}
In this work, we have developed an efficient LBM with an adaptive relaxation parameter and a Lax--Friedrichs-type equilibrium for hyperbolic systems. The equilibrium incorporates dissipation through characteristic-speed bounds, while a smoothness indicator based on characteristic projections adjusts the relaxation parameter locally. This combination directly addresses the compromise associated with a fixed relaxation parameter: the method approaches the second-order, low-dissipation limit in smooth regions and increases dissipation only near detected nonsmooth structures. The stabilization acts through the collision step and therefore preserves the standard collide-and-stream update without requiring nonlinear reconstruction or a Riemann-solver-based numerical flux.

Numerical analysis is conducted to clarify the accuracy and stability of our method. Maxwell iteration is employed to show its second-order accuracy in both space and time. For linear hyperbolic systems with periodic boundary conditions, the equilibrium Jacobians yield a matrix-valued weight under which the method is $L^2$-stable subject to the standard CFL condition. 
The proposed method is demonstrated through a variety of problems. The smooth linear-advection and Euler tests recover the expected asymptotic order. The discontinuous advection, shock-tube, shallow-water, 2D Riemann, and reactive-flow tests show that the relaxation parameter decreases near discontinuities while remaining close to $2$ in smooth regions. In the comparisons, the present LBM reduces the excessive dissipation of fixed $\omega=1$ and provides resolution and robustness comparable to the limiter-based vectorial LBM, with a lower measured cost in the Sod shock-tube test. In particular, the large-scale 2D cellular-detonation simulation resolves the corrugated leading front, transverse-wave interactions, and a persistent quasi-periodic cellular pattern over long times. These results support direct collision-based dissipation control as a simple and efficient stabilization mechanism for hyperbolic systems.

The rigorous stability analysis is presently restricted to linear hyperbolic systems with periodic boundary conditions. Extending the analysis to nonlinear systems and combining the method with invariant-domain or positivity-preserving mechanisms remain important topics for future work. Further extensions to large-scale 3D detonation waves and parallel multidimensional computations are also of interest.

%
%

\section*{Acknowledgments}
This work was supported by the Beijing Natural Science Foundation (No. JR25003) and the National Natural Science Foundation of China (No. 12301520).

\bibliographystyle{siamplain}
\bibliography{StabilityVLBMhyperbolic_gaiban}
\end{document}